\documentclass[reqno,12pt]{amsart}
\usepackage{times}
\usepackage{amsmath,amsfonts,amsthm,amscd,amssymb,graphicx,enumerate}
\usepackage{graphicx}

\makeatletter
\def\@strippedMR{}
\def\@scanforMR#1#2#3\endscan{
  \ifx#1M\ifx#2R\def\@strippedMR{#3}
  \else\def\@strippedMR{#1#2#3}
  \fi\fi}
\renewcommand\MR[1]{\relax\ifhmode\unskip\spacefactor3000 \space\fi
  \@scanforMR#1\endscan
  MR\MRhref{\@strippedMR}{\@strippedMR}}
\makeatother

\parskip=\medskipamount

\newtheorem*{Thm*}{Theorem}
\newtheorem{Thm}{Theorem}[section]
\newtheorem{Cor}[Thm]{Corollary}
\newtheorem{Prop}[Thm]{Proposition}
\newtheorem{Lemma}[Thm]{Lemma}

\newtheorem{theorem}[Thm]{Theorem}
\newtheorem{Corollary}[Thm]{Corollary}
\newtheorem{Proposition}[Thm]{Proposition}

\theoremstyle{definition}

\newtheorem{Notation}[Thm]{Notation}

\newtheorem{Remark}[Thm]{Remark}
\newtheorem{Example}[Thm]{Example}

\newcommand{\mf}[1]{\mathbb{#1}}
\newcommand{\mc}[1]{\mathcal{#1}}

\DeclareMathOperator{\NC}{\mathcal{NC}}

\DeclareMathOperator{\Inner}{\mathrm{Inn}}
\DeclareMathOperator{\Outer}{\mathrm{Out}}

\newcommand{\norm}[1]{\left\Vert#1\right\Vert}
\newcommand{\abs}[1]{\left\vert#1\right\vert}

\newcommand{\set}[1]{\left\{#1\right\}}
\newcommand{\ip}[2]{\left \langle #1, #2 \right \rangle}

\allowdisplaybreaks[1]

\title{Operator-valued Jacobi parameters and examples of operator-valued distributions}
\author{Michael Anshelevich}
\thanks{The first author was supported in part by NSF grant DMS-1160849.}
\address{Department of Mathematics, Texas A\&M University, College Station, TX 77843-3368}
\email{manshel@math.tamu.edu}
\author{John D. Williams}
\address{Universit{\"a}t des Saarlandes, Fachrichtung Mathematik.  Postfach 151150, 66041. Saarbr{\"u}cken}
\email{williams@math.uni-sb.de}
\subjclass[2010]{Primary 46L54; Secondary 42C05}
\date{\today}

\begin{document}

\begin{abstract}
In the setting of distributions taking values in a $C^\ast$-algebra $\mathcal{B}$, we define generalized Jacobi parameters and study distributions they generate. These include numerous known examples and one new family, of $\mathcal{B}$-valued free binomial distributions, for which we are able to compute free convolution powers.
Moreover, we develop a convenient combinatorial method for calculating the joint distributions of $\mathcal{B}$-free random variables with Jacobi parameters, utilizing two-color non-crossing partitions. This leads to several new explicit examples of free convolution computations in the operator-valued setting. Additionally, we obtain a counting algorithm for the number of two-color non-crossing pairings of relative finite depth, using only free probabilistic techniques. Finally, we show that the class of distributions with Jacobi parameters is not closed under free convolution.
\end{abstract}

\maketitle
\section{Introduction}

Let $\mu$ denote a probability measure on $\mathbb{R}$ all of whose moments are finite.  Then $\mu$ is associated to two numerical sequences
$\set{(\lambda_{i})_{i = 1}^\infty, (\alpha_{i})_{i =1}^\infty}$ where $\lambda_{i} , \alpha_{i} \in \mathbb{R}$ and $\alpha_{i} \geq 0$ for all $i \in \mathbb{N}$, the so-called \textit{Jacobi parameters} (see \cite{Chihara-book} for an overview).
The moments of the measure $\mu$ are calculated from these parameters using sums over Motzkin paths or non-crossing partitions, and $\mu$ has a moment generating function with the continued fraction expansion
\begin{equation}
\label{Continued-fraction}
M_{\mu}(z) = \frac{1}{1 - \lambda_{1}z - \frac{\alpha_{1}z^{2}}{1-\lambda_{2} - \frac{\alpha_{2}z^{2}}{\cdots}} }.
\end{equation}
See Section~\ref{Section:Scalar} for more details. Recall that $\mu$ can also be considered as a (positive-definite) linear functional on the algebra of polynomials $\mf{C}[X]$.

The study of $\mathcal{B}$-valued probability was initiated by Voiculescu in \cite{Voiculescu1}.
Let $\mathcal{B}$ denote a unital C$^{\ast}$-algebra and $X$ a self-adjoint symbol.  We define the \textit{non-commutative polynomials}
to be the algebraic free product of $\mathcal{B}$ and $X$.
Probability measures are replaced by non-commutative distributions, which are completely positive, $\mathcal{B}$-bimodular maps
$$ \mu : \mathcal{B}\langle X \rangle \mapsto \mathcal{B}.$$

When provided with appropriate notions of boundedness, $\mathcal{B}$-valued distributions may be realized in \textit{$\mc{B}$-valued probability spaces}, which are triples $(\mathcal{A}, E, \mathcal{B})$ with $\mathcal{B} \subset \mathcal{A}$ a unital containment of $C^{\ast}$ algebras and $E : \mathcal{A} \rightarrow \mathcal{B}$ a conditional expectation.  We say that random variables $X_{1} , X_{2} \in \mathcal{A}$ are $\mathcal{B}$-free if
$$E(P_{1}(X_{i_{1}}) P_{2}(X_{i_{2}}) \cdots P_{n}(X_{i_{n}}) ) = 0$$
whenever the $P_{j}(X) \in \mathcal{B}\langle X \rangle$ satisfy $E(P_{j}(X_{i_{j}})) = 0$ for all $j=1 , \ldots , n$ and $i_{1} \neq i_{2}, i_{2} \neq i_{3}, \cdots , i_{n-1} \neq i_{n}$.
If $X_{1}$ has distribution $\mu$ (that is $\mu(P(X)) = E(P(X_{1}))$ for all $P(X) \in \mathcal{B}\langle X \rangle$) and $X_{2}$ has distribution $\nu$
then we call the \textit{free convolution} of $\mu$ and $\nu$ the distribution of the random variable $X_{1} + X_{2}$.  In  symbols, this distribution is denoted $\mu \boxplus \nu$.

The goal of this article is to define $\mc{B}$-valued distributions associated to Jacobi parameters, and to study their properties. Here
in the $\mathcal{B}$-valued setting, Jacobi parameters will be pairs of sequences $\set{(\lambda_{i})_{i=1}^\infty, (\alpha_{i})_{i=1}^\infty}$
where $\lambda_{i} \in \mathcal{B}$ are self-adjoint elements and $\alpha_{i} \in \mathcal{CP(B)}$ are completely positive self maps of $\mathcal{B}$.

In Proposition \ref{Prop:Jacobi-definition}, we show that each such pair of sequences generates a $\mathcal{B}$-valued distribution, which we call a Jacobi-Szeg\H{o} distribution (to distinguish it from several other types of distributions named after Jacobi).
In Proposition \ref{Prop:Shift}, we recover the analog of the continued-fraction expansion \eqref{Continued-fraction}, and extend to this setting some familiar results for scalar-valued Jacobi parameters.
The remainder of section \ref{section2} is dedicated to constructing free Meixner distributions from Jacobi parameters.  This class contains most of the ``named'' distributions such as the semicircular, Bernoulli and free Poisson distributions. The key result is that free convolution powers of free Meixner distributions again belong to this class. This allows us to compute explicitly the $\mc{B}$-valued free binomial distributions, the free convolution powers of general $\mc{B}$-valued Bernoulli distributions. This computation is made even more explicit in the particular case treated in Proposition \ref{Bernoulli}.  This is notable as there are at present very few explicit computations of such convolutions in the literature.

In Section \ref{section3}, we show that the joint distribution of $\mc{B}$-free random variables, each of which has a Jacobi-Szeg\H{o} distribution, has a remarkably simple combinatorial structure based on certain subsets of $2$-color non-crossing partitions $\mathcal{TCNC}_{1,2}(n)$ (see Theorem \ref{maintheorem}). This is surprising since there is in general no direct relation between Jacobi parameters and freeness. In fact, our formula is new even in the scalar valued case (but see Example 2 in \cite{Ans-Product}, and \cite{Mlotkowski-Cumulants-Jacobi}, for related results).
Moreover, in Section \ref{section4}, we consider what is, in some sense, the class of atomic $\mathcal{B}$-valued distributions.  The formulas for the free convolution of two such distributions reduce to considering certain subsets $\mathcal{TCNC}_{1,2}^{k,\ell}(n) \subset \mathcal{TCNC}_{1,2}(n)$ where $k$ and $\ell$ refer to a constraint on the types of nesting that can occur in these pairings. Conversely, we obtain a recursive formula for the size of the sets $\mathcal{TCNC}_2^{k,k}(n) $ using free probabilistic methodology. It would be interesting to find a direct combinatorial argument for counting these sets.

Section \ref{section5} is concerned with  the consequences of these theorems. In Example \ref{simplest}, the simplest possible strictly $\mathcal{B}$-valued convolution operation is performed explicitly, through the Cauchy transform rather than combinatorial technology.  Lastly, Example \ref{counterexample} shows that the convolution of two $\mc{B}$-valued Bernoulli distributions is not in general a Jacobi-Szeg\H{o} distribution, providing a negative answer to a question posed by Roland Speicher.

\textbf{Acknowledgements.} The first author would like to thank all the co-authors of \cite{Ans-Bel-Fev-Nica}; a number of results in this article are a follow-up to our discussions. The authors would also like to thank Roland Speicher for numerous questions about operator-valued Jacobi parameters. Comments by a referee resulted in substantial improvements to the article, and are greatly appreciated.

\section{Preliminaries and scalar Jacobi parameters}
\label{Section:Scalar}

\subsection{Combinatorial preliminaries}

Let $\mathcal{NC}(n)$ be the collection of non-crossing partitions of the set $\set{1, 2, \ldots, n}$, and $\mathcal{NC}_{1,2}(n)$ the sub-collection of partitions whose blocks are singletons or pairs. For $\pi \in \NC(n)$, there is a natural partial order on its blocks: if $U, V \in \pi$, $U \succeq V$ if for some $a, b \in U$ and all $c \in V$, $a \leq c \leq b$. We say that $U$ covers $V$ if there is no $W \in \pi$ with $U \succ W \succ V$. The \emph{depth} of a block $V$ in $\pi$ is
\[
d(V, \pi) = \abs{\set{U \in \pi : U \succeq V}}.
\]
A block is outer if its depth is $1$, otherwise it is inner.

\subsection{Preliminaries on operator-valued distributions}

We briefly summarize the notions to be used throughout the paper. See, for example, \cite{Ans-Bel-Fev-Nica} for more details. Let $\mc{B}$ be a unital $C^\ast$-algebra. Denote by $\mc{CP}(\mc{B})$ the completely positive maps on $\mc{B}$, and by $\mc{B} \langle X \rangle$ the non-commutative polynomials with coefficients in $\mc{B}$. A \emph{non-commutative distribution} is a map $\mu : \mc{B} \langle X \rangle \rightarrow \mc{B}$ satisfying conditions (a,b) from
\begin{enumerate}
\item
$\mu$ is a $\mf{C}$-linear, unital, $\mc{B}$-bimodule map.
\item
$\mu$ is completely positive.
\item
$\mu$ is exponentially bounded, that is, there is a constant $M$ such that for any $b_0, b_1, \ldots, b_n \in \mc{B}$,
\[
\norm{\mu[b_0 X b_1 X \ldots X b_n]} \leq M^n \norm{b_0} \norm{b_1} \ldots \norm{b_n}.
\]
\end{enumerate}
We will denote the space of all non-commutative distributions by $\Sigma(\mc{B})$. The smaller set of exponentially bounded distributions, those satisfying (a-c), will be denoted by $\Sigma^0(\mc{B})$. Finally, algebraic non-commutative distributions are those satisfying only condition (a) above, and will be denoted by $\Sigma_{alg}(\mc{B})$. For $\mc{B} = \mf{C}$, these three classes correspond to, respectively, positive linear functionals on $\mf{C}[x]$, compactly supported measures on $\mf{R}$, and all linear functionals on $\mf{C}[x]$.

Let $(\mc{A}, E, \mc{B})$ be a non-commutative probability space, that is, $\mc{A}$ is a $C^\ast$-algebra containing $\mc{B}$ and $E : \mc{A} \rightarrow \mc{B}$ is a conditional expectation. Let $a \in \mc{A}$ be self-adjoint. Then \cite{Popa-Vinnikov-NC-functions}
\[
\mu[b_0 X b_1 X \ldots X b_n] = E[b_0 a b_1 a \ldots a b_n]
\]
is an exponentially bounded non-commutative distribution, and every such distribution arises in this way. The following proposition is also likely well-known, see \cite{Ans-Bel-Fev-Nica,SpeHab}, etc., but we do not have a precise reference.

\begin{Prop}
\label{Prop:Algebraic-distr}
Let $\ip{\cdot}{\cdot}$ be a (possibly degenerate) positive semi-definite $\mc{B}$-valued inner product on $\mc{B} \langle X \rangle$, and $x$ an operator on the corresponding pre-Hilbert bimodule symmetric with respect to this inner product, in the sense that for any $P, Q \in \mc{B} \langle X \rangle$, $\ip{x P}{Q} = \ip{P}{x Q}$. Then
\[
\mu[b_0 X b_1 X \ldots X b_n] = \ip{1_{\mc{B}}}{(b_0 x b_1 x \ldots x b_n) 1_{\mc{B}}}
\]
is a non-commutative distribution, and every such distribution arises in this way.
\end{Prop}

\begin{proof}
Bimodularity of $\mu$ follows from the definition of a $\mc{B}$-valued inner product. Complete positivity follows by observing that for any non-commutative polynomials $\set{P_i}_{i=1}^n$,
\begin{equation}
\label{Positivity}
\sum_{i, j=1}^n b_i^\ast \mu[P_i^\ast(X) P_j(X)] b_j = \ip{\sum_i P_i(x) b_i}{\sum_j P_j(x) b_j} \geq 0,
\end{equation}
where we have used the symmetry of $x$. For the converse, suppose $\mu$ is a non-commutative distribution. On $\mc{B} \langle X \rangle$, define the $\mc{B}$-valued inner product by a $\mf{C}$-linear extension of
\[
\ip{b_0 X b_1 X \ldots X b_n}{c_0 X c_1 X \ldots X c_k}
= \delta_{nk} \mu[b_n^\ast X \ldots X b_1^\ast X b_0^\ast c_0 X c_1 X \ldots X c_k].
\]
The $\mc{B}$-valued inner product property
\[
\ip{b P b'}{Q b''} = (b')^\ast \ip{P}{b^\ast Q} b''
\]
follows from bimodularity of $\mu$. Positive semi-definiteness follows by reversing equation~\eqref{Positivity}. If we take $x$ to be the operator of multiplication by $X$, its symmetry also follows directly from the definition.
\end{proof}

\begin{Notation}
For $\mu \in \Sigma(\mc{B})$, define its the moment generating function, considered as a formal power series, by
\[
M_\mu(b) = \sum_{n=0}^\infty \mu \left[ (X b)^n \right],
\]
The Cauchy transform of $\mu$ is
\[
G_\mu(b) = b^{-1} M_\mu(b^{-1}).
\]
For $\mu \in \Sigma^0(\mc{B})$ with bound $M$, $G_\mu$ can be identified with a non-commutative analytic function, and $M_\mu(b)$ is a convergent series for $\norm{b} \leq M^{-1}$. The free cumulant generating function of $\mu$ is defined, as a formal power series, implicitly via
\begin{equation}
\label{Free-cumulant-GF}
M_\mu(b) = 1_{\mc{B}} + R_\mu(b M_\mu(b))
\end{equation}
(compare with Corollary 5.4 in \cite{Popa-Vinnikov-NC-functions}). Occasionally, we will also use the Boolean cumulant generating function
\begin{equation}
\label{Boolean}
B_\mu(b) = 1 - (M_\mu(b))^{-1},
\end{equation}
and Boolean convolution, which satisfies
\begin{equation}
\label{Boolean-conv}
B_{\mu \uplus \nu}(b) = B_\mu(b) + B_\nu(b),
\end{equation}
and can be defined using the fully matricial version of this identity (see Proposition~\ref{Prop:Fully-matricial}). A key result in \cite{Ans-Bel-Fev-Nica} is that for any $\eta \in \mc{CP}(\mc{B})$, one can define Boolean and free convolution powers $\mu^{\uplus \eta}$ and $\mu^{\boxplus \eta}$.
\end{Notation}

\subsection{Scalar Jacobi parameters}

We recall the following fundamental theorem. See \cite{Flajolet,Viennot-Notes,AccBozGaussianization} for details, and further possible equivalences.

\begin{Thm}
\label{Thm;Favard}
Consider two sequences $\set{(\lambda_{i} \in \mf{R})_{i=1}^\infty, (\alpha_{i} > 0)_{i=1}^\infty}$, and a probability measure $\mu$ on $\mf{R}$ all of whose moments are finite. The following are equivalent.
\begin{enumerate}
\item
The moment generating function of $\mu$ has a continued fraction expansion
\[
M(z) = \sum_{n=0}^\infty z^n \mu[x^n] =
\dfrac{1}{1 - \lambda_1 z -
\dfrac{\alpha_1 z^2}{1 - \lambda_2 z -
\dfrac{\alpha_2 z^2}{1 - \ldots}}}.
\]
\item
$\mu$ is the distribution of the tridiagonal matrix
\[
\begin{pmatrix}
\lambda_1 & \alpha_1 & 0 & 0 & \ldots \\
1 & \lambda_2 & \alpha_2 & 0 & \ldots \\
0 & 1 & \lambda_3 & \alpha_3 & \ldots \\
0 & 0 & 1 & \lambda_4 & \ddots \\
\vdots & \vdots & \ddots & \ddots & \ddots
\end{pmatrix}
\]
with respect to the vector state corresponding to the top left entry of the matrix.
\item
For any $n$
\begin{equation*}
\mu[x^n] = \sum_{\pi \in \mathcal{NC}_{1,2}(n)}
\prod_{\substack{V \in \pi \\ |V| = 1}} \lambda_{d(V,\pi)} \cdot
\prod_{\substack{V \in \pi \\ |V|=2}} \alpha_{d(V,\sigma)}.
\end{equation*}
\end{enumerate}
In addition, finite sequences of the form $\set{(\lambda_{i} \in \mf{R})_{i=1}^{k+1}, (\alpha_{i} > 0)_{i=1}^k}$ for some $k$ correspond to finitely supported measures, terminating continued fractions, finite tridiagonal matrices, and sums over partitions $\mathcal{NC}^k_{1,2}(n)$ defined in Section~\ref{section4}.
\end{Thm}

\section{Jacobi parameters and continued fraction expansions}\label{section2}

\begin{Prop}
\label{Prop:Jacobi-definition}
Let $\set{\lambda_{i} \in \mc{B}}_{i=1}^\infty$ be self-adjoint, and $\set{\alpha_{i} \in \mc{CP}(\mc{B})}_{i=1}^\infty$. On the vector space $\mc{B} \langle X \rangle$, define the $\mc{B}$-valued inner product
\begin{equation}
\label{Inner-product}
\ip{b_0 X b_1 X \ldots X b_n}{c_0 X c_1 X \ldots X c_k}
= \delta_{nk} b_n^\ast \alpha_1 \Bigl[b_{n-1}^\ast \alpha_2 \bigl[\ldots \alpha_n [b_0^\ast c_0] c_1 \bigr] \ldots c_{n-1}\Bigr] c_n,
\end{equation}
in particular $\ip{b}{c} = b^\ast c$. This inner product may be degenerate, but we will only use it to compute moments. On the induced pre-Hilbert bimodule, define operators
\[
a^\ast(b_0 X b_1 X \ldots X b_n) = X b_0 X b_1 X \ldots X b_n,
\]
\[
p(b_0 X b_1 X \ldots X b_n) = \lambda_n b_0 X b_1 X \ldots X b_n,
\]
\[
a(b_0 X b_1 X \ldots X b_n) = \alpha_n[b_0] b_1 X \ldots X b_n,
\]
$a(b) = 0$, and
\[
x = a^\ast + p + a.
\]
Then $p$ and $a^\ast + a$, and so also $x$, are symmetric. Therefore $\mu : \mc{B} \langle X \rangle \rightarrow \mc{B}$ defined as in Proposition~\ref{Prop:Algebraic-distr}
is a non-commutative distribution.

We denote
\begin{equation}
\label{Jacobi-parameters}
\mu = J
\begin{pmatrix}
\lambda_1, & \lambda_2, & \lambda_3, & \lambda_4, & \ldots \\
\alpha_1, & \alpha_2, & \alpha_3, & \alpha_4, & \ldots
\end{pmatrix}
\end{equation}
and refer to it as the \emph{Jacobi-Szeg\H{o} distribution} with Jacobi parameters $\set{(\lambda_i)_{i=1}^\infty, (\alpha_i)_{i=1}^\infty}$.
\end{Prop}

\begin{proof}
Clearly $p$ is symmetric, and
\[
\begin{split}
\ip{a^\ast(b_0 X b_1 X \ldots X b_n)}{c_0 X c_1 X \ldots c_n X c_{n+1}}
& = b_n^\ast \alpha_1 \Bigl[b_{n-1}^\ast \alpha_2 \bigl[ \ldots b_0^\ast \alpha_{n+1} [1_{\mc{B}} c_0] c_1 \bigr] \ldots c_n \Bigr] c_{n+1} \\
& = \ip{b_0 X b_1 X \ldots X b_n}{\alpha_{n+1} [c_0] c_1 X \ldots c_n X c_{n+1}} \\
& = \ip{b_0 X b_1 X \ldots X b_n}{a(c_0 X c_1 X \ldots c_n X c_{n+1})},
\end{split}
\]
so $a^\ast + a$ is also symmetric. Thus, $x$ is symmetric, and $\mu$ is its distribution with respect to a vector state.
\end{proof}

\begin{Example}
The values of
\begin{equation*}
\mu = J
\begin{pmatrix}
\lambda_1, & \lambda_2, & \lambda_3, & \lambda_4, & \ldots \\
\alpha_1, & \alpha_2, & \alpha_3, & \alpha_4, & \ldots
\end{pmatrix}
\end{equation*}
on low-degree polynomials are
\[
\mu[b_0] = b_0,
\]
\[
\mu[b_0 X b_1] = b_0 \lambda_1 b_1,
\]
\[
\mu[b_0 X b_1 X b_2] = b_0 \lambda_1 b_1 \lambda_1 b_2 + b_0 \alpha_1[b_1] b_2.
\]
\end{Example}

\begin{Prop}\label{basejacobi}
Let $\pi \in \NC_{1,2}$. Define $T_\pi(b_0 X \ldots X b_n)$ as follows. Consider $\pi$ as a partition of the set of $n$ $X$'s in its argument. If a single $X$ is a block, it is replaced by a $\lambda$. If a pair of $X$'s form a block, they are replaced by an application of an $\alpha$ to the terms between these $X$'s. In each case, the index of $\lambda$ or $\alpha$ is the depth of the block in $\pi$.  For example, for $\pi = \{ (1,4) , (2,3) \}$,  $\pi' =  \{ (1,2) , (3,4) \}$ and $\pi'' = \{ (1,3) , (2) , (4)\}$
we have \begin{align*}
T_{\pi}(Xb_{1}Xb_{2}Xb_{3}X) & = \alpha_{1}[b_{1}\alpha_{2}[b_{2}] b_{3}] \\
 T_{\pi'}(Xb_{1}Xb_{2}Xb_{3}X) & = \alpha_{1}[b_{1}] b_{2} \alpha_{1}[b_{3}]  \\
T_{\pi''}(Xb_{1}Xb_{2}Xb_{3}X) & = \alpha_{1}[b_{1} \lambda_{1}b_{2}] b_{3}\lambda_{0}.
\end{align*}
See Remark~3.2 in \cite{Ans-Bel-Fev-Nica} for a detailed description of a similar construction. Then we get the following extension of part (c) of Theorem~\ref{Thm;Favard}:
\begin{equation}
\label{Moments}
\mu[b_0 X \ldots X b_n]
= \sum_{\pi \in \NC_{1,2}(n)} T_\pi(b_0 X \ldots X b_n).
\end{equation}
\end{Prop}

\begin{proof}
The argument is similar to the scalar-valued case \cite{AccBozGaussianization} and the operator-valued semicircular case \cite{SpeHab}, so we only provide an outline. In the notation of Propositions~\ref{Prop:Algebraic-distr} and \ref{Prop:Jacobi-definition},
\[
\mu[b_0 X b_1 X \ldots X b_n] = \ip{1_{\mc{B}}}{(b_0 x b_1 x \ldots x b_n) 1_{\mc{B}}}.
\]
Since $x = a^\ast + p + a$, this can be decomposed as sum of $3^n$ terms of the form
\[
W_u = \ip{1_{\mc{B}}}{(b_0 u_1 b_1 u_2 \ldots u_n b_n) 1_{\mc{B}}},
\]
where each $u_i$ is one of $a^\ast, p, a$. $\mc{B} \langle X \rangle$ is graded, with $a^\ast$ increasing the grading by $1$, $p$ preserving the grading, and $a$ decreasing the grading by $1$, and different components in the grading are orthogonal with respect to the inner product \eqref{Inner-product}. It follows that out of the $3^n$ terms above, most are zero, and each of the remaining ones arises from a $\pi \in \NC_{1,2}(n)$ as follows: if $\set{i,j} \in \pi$ with $i < j$, then $u_i = a$, $u_j = a^\ast$; and if $\set{i} \in \pi$, then $u_i = p$. Moreover, it follows from the definitions of $a^\ast, p, a$ that $W_u = T_\pi(b_0 X \ldots X b_n)$.
\end{proof}

\begin{Remark}
\label{Remark:mu_}
Let $\lambda \in \mc{B}$ be self-adjoint, and $\beta: \mc{B} \langle X \rangle \rightarrow \mc{B}$ be a $\mf{C}$-linear, completely positive (but not necessarily $\mc{B}$-bimodule) map. Out of this data, in Lemmas~7.2, 7.3 of \cite{Ans-Bel-Fev-Nica} was constructed a non-commutative distribution $\mu_{(\lambda, \beta)}$ such that
\begin{equation}
\label{Jacobi-sum-1}
\begin{split}
\mu[b_0 X \ldots X b_n] = \sum_{k=1}^n \sum_{1 \leq i_1 < i_2 < \ldots < i_k = n}
& b_0 \beta[b_1 X \ldots b_{i_1 - 1}] b_{i_1} \beta[b_{i_1 + 1} X \ldots b_{i_2 - 1}] b_{i_2} \\
& \ldots b_{i_{k-1}} \beta[b_{i_{k-1} + 1} X \ldots b_{n-1}] b_n,
\end{split}
\end{equation}
where $\beta[\emptyset] = \lambda$. It follows that the Boolean cumulant functional of $\mu$ is $B_\mu[X] = \lambda$,
\[
B_\mu[X b_1 X \ldots X b_{n-1} X] = \beta[b_1 X \ldots X b_{n-1}]
\]
(for the reader unfamiliar with the notion of Boolean cumulants, this relation can be taken as their definition; see \cite{Popa-Vinnikov-NC-functions} for more details).
\end{Remark}

In the following proposition, the map $\mu_n \mapsto \mu_{n+1}$ is sometimes called ``coefficient stripping''.

\begin{Prop}
\label{Prop:Shift}
Denote
\[
\mu_n = J
\begin{pmatrix}
\lambda_{n}, & \lambda_{n+1}, & \lambda_{n+2}, & \ldots \\
\alpha_n, & \alpha_{n+1}, & \alpha_{n+2}, & \ldots
\end{pmatrix}.
\]
Then in the notation of Remark~\ref{Remark:mu_},
\[
\mu_n = \mu_{(\lambda_n, \beta_n)},
\]
where $\beta_n = \alpha_n \circ \mu_{n+1}$. Also, the moment generating function of $\mu$ has a continued fraction expansion
\[
M_\mu(b) = \left( 1_{\mc{B}} - \lambda_1 b - \alpha_1 \left[ b \left( 1_{\mc{B}} - \lambda_2 b - \alpha_2 [b \ldots] b \right)^{-1} \right] b \right)^{-1}.
\]
More precisely, in the expansions of $M_\mu(b)$ and of a finite continued fraction
\begin{equation}
\label{Continued-fraction-k}
\left( 1_{\mc{B}} - \lambda_1 b - \alpha_1 \left[ b \left( 1_{\mc{B}} - \lambda_2 b - \alpha_2 \left[\ldots b \left( 1_{\mc{B}} - \lambda_{k} b - \alpha_{k} [b] b \right)^{-1} \right] b \right)^{-1} \right] b \right)^{-1}.
\end{equation}
in formal power series in $b$, the first $k$ terms coincide, so these finite continued fractions converge to $M_\mu(b)$ as formal series.
\end{Prop}

\begin{proof}
By collecting in formula~\eqref{Moments} the terms with $\lambda_i, \alpha_i$ with $i \geq 2$, we obtain
\begin{equation}
\label{Jacobi-sum}
\begin{split}
\mu[b_0 X \ldots X b_n] = \sum_{k=1}^n \sum_{1 \leq i_1 < i_2 < \ldots < i_k = n}
& b_0 \alpha_1 \Bigl[\mu_2[b_1 X \ldots b_{i_1 - 1}] \Bigr] b_{i_1} \alpha_1 \Bigl[\mu_2[b_{i_1 + 1} X \ldots b_{i_2 - 1}] \Bigr] b_{i_2} \\
& \ldots b_{i_{k-1}} \alpha_1 \Bigl[\mu_2[b_{i_{k-1} + 1} X \ldots b_{n-1}] \Bigr] b_n,
\end{split}
\end{equation}
where $\alpha_1[\emptyset] = \lambda_1$. So if $\mu = \mu_{(\lambda, \beta)}$, comparing this with the combinatorial formula~\eqref{Jacobi-sum-1}, we see that $\lambda = \lambda_1$ and $\beta = \alpha_1 \circ \mu_2$. The result for $\mu_n$ follows by induction.

From equation \eqref{Jacobi-sum},
\[
\mu \left[ \sum_{n=0}^\infty (X b)^n \right]
= 1_{\mc{B}} + \sum_{k=1}^\infty \left( \lambda_1 b + \alpha_1 \Bigl[ b \mu_2 \Bigl[ \sum_{n=0}^\infty (X b)^n \Bigr] \Bigr] b \right)^k
\]
and so
\begin{equation}
\label{mu-mu2}
M_\mu(b) = \left( 1_{\mc{B}} - \lambda_1 b - \alpha_1 [b M_{\mu_2}(b)] b \right)^{-1}.
\end{equation}
Iterating, we obtain
\[
M_\mu(b) =
\left( 1_{\mc{B}} - \lambda_1 b - \alpha_1 \left[ b \left( 1_{\mc{B}} - \lambda_2 b - \alpha_2 \left[\ldots b \left( 1_{\mc{B}} - \lambda_{k} b - \alpha_{k} [b M_{\mu_{k+1}}(b)] b \right)^{-1} \right] b \right)^{-1} \right] b \right)^{-1}.
\]
This implies the equality of the first $k$ terms in the expansions of $M_\mu(b)$ and \eqref{Continued-fraction-k}.
\end{proof}

\begin{Cor}
If all $\norm{\lambda_i}_{i=1}^\infty$, $\norm{\alpha_i}_{i=1}^\infty$ are uniformly bounded, then
\begin{equation*}
\mu = J
\begin{pmatrix}
\lambda_1, & \lambda_2, & \lambda_3, & \lambda_4, & \ldots \\
\alpha_1, & \alpha_2, & \alpha_3, & \alpha_4, & \ldots
\end{pmatrix}
\end{equation*}
is an exponentially bounded non-commutative distribution.
\end{Cor}

\begin{proof}
If $M$ is the uniform bound, it suffices to note that $\abs{\NC_{1,2}(n)} \leq 4^n$ and for any $\pi$,
\[
\norm{T_\pi(b_0 X \ldots X b_n)} \leq M^n \norm{b_0} \norm{b_1} \ldots \norm{b_n}. \qedhere
\]
\end{proof}

\begin{Corollary}
In the setting of the preceding corollary, the convergence of the continued fraction approximants in Proposition~\ref{Prop:Shift} is in norm pointwise for $\norm{b} \leq M^{-1}$.
\end{Corollary}

\begin{proof}
Denote
\[
\mu^{(k)} = J
\begin{pmatrix}
\lambda_{1}, & \lambda_{2}, & \ldots, & \lambda_{k}, & \lambda_{k+1}, & 0, & 0, & \ldots \\
\alpha_1, & \alpha_{2}, & \ldots, & \alpha_{k}, & 0, & 0, & 0, & \ldots
\end{pmatrix}.
\]
Then both $\mu$ and all $\mu^{(k)}$ are exponentially bounded with constant $M$, and so the series defining $M_\mu(b)$, $M_{\mu^{(k)}}(b)$ converge for $\norm{b} < M^{-1}$. Moreover by Proposition~\ref{Prop:Shift}, for each $k$, the first $k$ terms of these series coincide. It follows that $M_{\mu^{(k)}}(b) \rightarrow M_\mu(b)$ in norm.
\end{proof}

\begin{Prop}
\label{Prop:Fully-matricial}
Let
\begin{equation*}
\mu = J
\begin{pmatrix}
\lambda_1, & \lambda_2, & \lambda_3, & \lambda_4, & \ldots \\
\alpha_1, & \alpha_2, & \alpha_3, & \alpha_4, & \ldots
\end{pmatrix}.
\end{equation*}
Fix $d \in \mf{N}$. Define $\widetilde{\alpha}_i = I_d \otimes \alpha_i$ to be the map on $M_d(\mf{C}) \otimes \mc{B} \simeq M_d(\mc{B})$ and $\widetilde{\lambda}_i$ to be a self-adjoint element $1_d \otimes \lambda_i \in M_d(\mf{C}) \otimes \mc{B}$. Also define $\widetilde{\mu}$ to be the $M_d(\mf{C}) \otimes \mc{B}$-bimodule map
\[
I_d \otimes \mu : M_d(\mc{B}) \langle X \rangle \rightarrow M_d(\mc{B}).
\]
The family of $\widetilde{\mu}$ for $d \in \mf{N}$ is the \emph{fully matricial extension} of $\mu$.
\begin{enumerate}
\item
$\widetilde{\mu}$ is also a Jacobi-Szeg\H{o} distribution, with Jacobi parameters
\begin{equation*}
\widetilde{\mu} = J
\begin{pmatrix}
\widetilde{\lambda}_0, & \widetilde{\lambda}_1, & \widetilde{\lambda}_2, & \widetilde{\lambda}_3, & \ldots \\
\widetilde{\alpha}_1, & \widetilde{\alpha}_2, & \widetilde{\alpha}_3, & \widetilde{\alpha}_4, & \ldots
\end{pmatrix}.
\end{equation*}
\item
The collection of all $M_{\widetilde{\mu}}$ for $d \in \mf{N}$ determines $\mu$.
\end{enumerate}
In the formulas below, we will thus prove the results for $d=1$ and conclude that they hold for general $d$, and so determine $\mu$.
\end{Prop}

\begin{proof}
Part (b) is standard, see for example \cite{Popa-Vinnikov-NC-functions}. The proof of part (a) parallels that of Proposition~6.3 of \cite{Popa-Vinnikov-NC-functions}, where a similar result is proved for free and Boolean cumulants, so we only give an outline. Denote $T_\pi$ the expression from Proposition~\ref{basejacobi} for $\set{(\lambda_i)_{i=1}^\infty, (\alpha_i)_{i=1}^\infty}$, and $\widetilde{T}_\pi$ the corresponding expression for $\set{(\widetilde{\lambda}_i)_{i=1}^\infty, (\widetilde{\alpha}_i)_{i=1}^\infty}$. Also, let $(e_{ij})_{i,j=1}^d$ be matrix units. Then
\[
\begin{split}
\widetilde{\mu}(b_0 e_{i(0), j(0)} X b_1 e_{i(1), j(1)} \ldots X b_n e_{i(n), j(n)})
& = e_{i(0), j(0)} e_{i(1), j(1)} \ldots e_{i(n), j(n)} \mu(b_0 X b_1 \ldots X b_n) \\
& = e_{i(0), j(0)} e_{i(1), j(1)} \ldots e_{i(n), j(n)} \sum_{\pi \in \NC(n)} T_\pi(b_0 X b_1 \ldots X b_n) \\
& = \sum_{\pi \in \NC(n)} \widetilde{T}_\pi(b_0 e_{i(0), j(0)} X b_1 e_{i(1), j(1)} \ldots X b_n e_{i(n), j(n)}).
\end{split}
\]
By linearity, it follows that $\widetilde{\mu} =  \sum_{\pi \in \NC(n)} \widetilde{T}_\pi$. 
\end{proof}

The following result is well-known in the scalar case, see for example \cite{Boz-Wys}.

\begin{Prop}
\label{Prop:Boolean}
The Jacobi parameters of a Boolean convolution power of
\[
\mu = J
\begin{pmatrix}
\lambda_1, & \lambda_2, & \lambda_3, & \lambda_4, & \ldots \\
\alpha_1, & \alpha_2, & \alpha_3, & \alpha_4, & \ldots
\end{pmatrix}
\]
are
\[
\mu^{\uplus \eta} = J
\begin{pmatrix}
\eta[\lambda_1], & \lambda_2, & \lambda_3, & \lambda_4, & \ldots \\
\eta \circ \alpha_1, & \alpha_2, & \alpha_3, & \alpha_4, & \ldots
\end{pmatrix}.
\]
\end{Prop}

\begin{proof}
According to Theorem~7.5 of \cite{Ans-Bel-Fev-Nica},
\[
\mu^{\uplus \eta}_{(\lambda_1, \beta_0)}
= \mu_{(\eta[\lambda_1], \eta \circ \beta_0)}
= \mu_{(\eta[\lambda_1], (\eta \circ \alpha_1) \circ \mu_1)}.
\]
The result follows from Proposition~\ref{Prop:Shift}.
\end{proof}

The remainder of the section treats examples of specific $\mc{B}$-valued Jacobi-Szeg\H{o} distributions.

\begin{Prop}
For $\lambda \in \mc{B}$ self-adjoint, the atomic distribution $\delta_\lambda$ has Jacobi parameters
\[
\mu = J
\begin{pmatrix}
\lambda, & 0, & 0, & 0, & \ldots \\
0, & 0, & 0, & 0, & \ldots
\end{pmatrix}.
\]
\end{Prop}

\begin{proof}
If $\mu$ is the distribution with these Jacobi parameters, then
\[
M_{\widetilde{\mu}}(b) = (1_{\mc{B}} - {\widetilde{\lambda}} b)^{-1},
\]
so that
\[
\mu[X b_1 X \ldots X b_n] = \lambda b_1 \lambda \ldots \lambda b_n
\]
and $\mu[P] = P(\lambda)$ for any $P \in \mc{B} \langle X \rangle$.
\end{proof}

\begin{Prop}
\label{Prop:Mean}
Let $\lambda \in \mc{B}$ be self-adjoint, and
\begin{equation*}
\mu = J
\begin{pmatrix}
\lambda_1, & \lambda_2, & \lambda_3, & \lambda_4, & \ldots \\
\alpha_1, & \alpha_2, & \alpha_3, & \alpha_4, & \ldots
\end{pmatrix}.
\end{equation*}
Then
\[
\mu \boxplus \delta_{\lambda}  = J
\begin{pmatrix}
\lambda_1 + \lambda, & \lambda_2 + \lambda, & \lambda_3 + \lambda, & \lambda_4 + \lambda, & \ldots \\
\alpha_1, & \alpha_2, & \alpha_3, & \alpha_4, & \ldots
\end{pmatrix}.
\end{equation*}
\end{Prop}

\begin{proof}
Let $X$ be an operator in a $\mc{B}$-valued probability space $(\mc{A}, E, \mc{B})$ (see the Introduction). Then directly from the definition of freeness, $\lambda$ and $X$ are $\mc{B}$-free, so that $\mu \boxplus \delta_{\lambda}$ is the distribution of $X + \lambda$. Thus
\[
\begin{split}
M_{\mu \boxplus \delta_{\lambda}}(b)
& = \sum_{n=0}^\infty \mu\left[ ((X + \lambda) b)^n \right] \\
& = \sum_{k=0}^\infty (1 - \lambda b)^{-1} \mu \left[ (X b (1 - \lambda b)^{-1})^k \right]
= (1 - \lambda b)^{-1} M_\mu(b (1 - \lambda b)^{-1}).
\end{split}
\]
Plugging this into equation~\eqref{mu-mu2}, we get
\[
\begin{split}
M_{\mu \boxplus \delta_{\lambda}}(b)
& = (1 - \lambda b)^{-1} \Bigl( 1 - \lambda_1 b (1 - \lambda b)^{-1} - \alpha_1 \left[ b (1 - \lambda b)^{-1} M_{\mu_2}(b (1 - \lambda b)^{-1}) \right] b (1 - \lambda b)^{-1} \Bigr)^{-1} \\
& = \Bigl( 1 - (\lambda_1 + \lambda) b - \alpha_1 \left[ b M_{\mu_2 \boxplus \delta_{\lambda}}(b) \right] b \Bigr)^{-1}.
\end{split}
\]
Repeating this calculation for $\mu_2, \mu_3$, etc., in the fully matrical setting of Proposition~\ref{Prop:Fully-matricial}, we obtain the conclusion.
\end{proof}

\begin{Prop}
For $\alpha \in \mc{CP}(\mc{B})$, the semicircular distribution with covariance $\alpha$ has Jacobi parameters
\[
\mu = J
\begin{pmatrix}
0, & 0, & 0, & 0, & \ldots \\
\alpha, & \alpha, & \alpha, & \alpha, & \ldots
\end{pmatrix}.
\]
\end{Prop}

\begin{proof}
If $\mu$ is the distribution with these Jacobi parameters, then
\begin{equation}
\label{Semicircular1}
M_{\widetilde{\mu}}(b) = \left( 1_{\mc{B}} - {\widetilde{\alpha}}[b M_{\widetilde{\mu}}(b)] b \right)^{-1},
\end{equation}
or equivalently
\[
M_{\widetilde{\mu}}(b) = 1_{\mc{B}} + {\widetilde{\alpha}}[b M_{\widetilde{\mu}}(b)] b M_{\widetilde{\mu}}(b).
\]
In terms of the Cauchy transform, this says
\[
b G_{\widetilde{\mu}}(b) = 1_{\mc{B}} + {\widetilde{\alpha}}[G_{\widetilde{\mu}}(b)] G_{\widetilde{\mu}}(b),
\]
which is equation (1.2) from \cite{Helton-Far-Speicher} (with $\eta$ from that paper being our $\alpha$). So $\mu$ is the centered $\mc{B}$-valued semicircular distribution with covariance $\alpha$. Note also that its free cumulant generating function is $R_\mu(b) = \alpha[b] b$, as it should be.
\end{proof}

\begin{Example}
\label{Ex:Bernoulli}
For $\lambda_1, \lambda_2 \in \mc{B}$ self-adjoint, and $\alpha \in \mc{CP}(\mc{B})$, we define a general $\mc{B}$-valued Bernoulli distribution via its Jacobi parameters
\[
\mu = J
\begin{pmatrix}
\lambda_1, & \lambda_2, & 0, & 0, & \ldots \\
\alpha, & 0, & 0, & 0, & \ldots
\end{pmatrix}.
\]
More explicitly,
\[
M_\mu(b) = \left( 1_{\mc{B}} - \lambda_1 b - \alpha \left[ b \left( 1_{\mc{B}} - \lambda_2 b \right)^{-1} \right] b \right)^{-1}.
\]
The name is justified by two particular cases. First, if all $\lambda_i \equiv 0$, then
\[
M_{\widetilde{\mu}}(b) = \left( 1_{\mc{B}} - {\widetilde{\alpha}} \left[ b \right] b \right)^{-1}.
\]
Comparing this with Corollary 2.2 from \cite{Belinschi-Popa-Vinnikov-Semicircle} (with slightly different notation), we see that $\mu$ is the centered $\mc{B}$-valued Bernoulli law with covariance $\alpha$. The second particular case is given in the following proposition.
\end{Example}

\begin{Prop}
For $0 < t < 1$ and $a, c \in \mc{B}^{sa}$, the distribution
\[
t \delta_a + (1-t) \delta_c
\]
is of the form in the preceding example, with
\begin{align*}
\lambda_1 &= t a + (1-t) c, \\
\lambda_2 & = (1-t) a + t c, \\
\alpha[b] & = t (1-t) (a - c) b (a-c).
\end{align*}
\end{Prop}

\begin{proof}
By translation, it suffices to prove this for $c=0$. So let $\lambda_1 = t a$, $\lambda_2 = (1-t) a$, $\alpha[b] = t (1-t) a b a$. Then
\[
\begin{split}
& \Bigl((1-t) \delta_0 + t \delta_a \Bigr) [b_0 X \ldots X b_n]
= t b_0 a \ldots a b_n
= b_0 (t + (1-t)) a \ldots (t + (1-t)) a b_{n-1} t a b_n \\
&\qquad = \sum_{k=1}^n \sum_{1 \leq i_1 < i_2 < \ldots < i_k = n}
b_0 \Bigl((1-t) a b_1 (1-t) a b_2 \ldots t a \Bigr) b_{i_1} \\
&\qquad\qquad \Bigl((1-t) a b_{i_1 + 1} (1-t) a b_{i_1 + 2} \ldots t a \Bigr) b_{i_2} \ldots \Bigl((1-t) a b_{i_{k-1} + 1} (1-t) a b_{i_{k-1} + 2} \ldots t a \Bigr) b_n \\
&\qquad = \sum_{k=1}^n \sum_{1 \leq i_1 < i_2 < \ldots < i_k = n}
b_0 \Bigl(a b_1 (1-t) a b_2 \ldots t (1-t) a \Bigr) b_{i_1} \\
&\qquad\qquad \Bigl(a b_{i_1 + 1} (1-t) a b_{i_1 + 2} \ldots t (1-t) a \Bigr) b_{i_2} \ldots \Bigl(a b_{i_{k-1} + 1} (1-t) a b_{i_{k-1} + 2} \ldots t (1-t) a \Bigr) b_n \\
&\qquad = \sum_{k=1}^n \sum_{1 \leq i_1 < i_2 < \ldots < i_k = n}
b_0 \alpha[b_1 \lambda_2 b_2 \ldots b_{i_1 -1}] b_{i_1} \\
&\qquad\qquad \alpha[b_{i_1 + 1} \lambda_2 b_{i_1 + 2} \ldots b_{i_2 - 1}] b_{i_2} \ldots \alpha[b_{i_{k-1} + 1} \lambda_2 b_{i_{k-1} + 2} \ldots b_{n-1}] b_n.
\end{split}
\]
where $\alpha[\emptyset] = t a = \lambda_1$, which is precisely formula~\eqref{Jacobi-sum} for $\lambda_1, \lambda_2, \alpha$ as above and $\mu_2 = \delta_{\lambda_2}$.
\end{proof}


\begin{Prop}
The centered free Poisson distribution with parameters $(\lambda, \alpha)$ has Jacobi parameters
\[
\mu = J
\begin{pmatrix}
0, & \lambda, & \lambda, & \lambda, & \ldots \\
\alpha, & \alpha, & \alpha, & \alpha, & \ldots
\end{pmatrix}.
\]
\end{Prop}

\begin{proof}
If $\mu$ is the distribution with these Jacobi parameters, then
\begin{equation}
\label{Shift}
M_{\mu}(b) = \left( 1_{\mc{B}} - \alpha \left[ b \left( 1_{\mc{B}} - \lambda b - \alpha [b \ldots] b \right)^{-1} \right] b \right)^{-1}
= \left( 1_{\mc{B}} - \alpha \left[ b M_{\nu}(b) \right] b \right)^{-1},
\end{equation}
where $\nu$ is a semicircular distribution with mean $\lambda$ and covariance $\alpha$. So
\begin{equation}
\label{Semicircular}
M_{\nu}(b) = \left( 1_{\mc{B}} - \lambda b - \alpha \left[ b M_{\nu}(b) \right] b \right)^{-1}.
\end{equation}
Combining equations \eqref{Shift} and \eqref{Semicircular}, we get
\[
M_\nu(b) = M_\mu(b) (1_{\mc{B}} - \lambda b M_\mu(b))^{-1},
\]
so
\[
M_\mu(b) = \left( 1_{\mc{B}} - \alpha \left[ b M_\mu(b) (1_{\mc{B}} - \lambda b M_\mu(b))^{-1} \right] b \right)^{-1}.
\]
Thus
\[
M_\mu(b) = 1_{\mc{B}} + \alpha \left[ b M_\mu(b) (1_{\mc{B}} - \lambda b M_\mu(b))^{-1} \right] b M_\mu(b).
\]
So
\[
R_\mu(b) = \alpha[b (1_{\mc{B}} - \lambda b)^{-1}] b = \sum_{n=0}^\infty \alpha[b (\lambda b)^n].
\]
By applying the arguments above to $\widetilde{\mu}$ as in Proposition~\ref{Prop:Fully-matricial}, and comparing with Definition~9.3 in \cite{Ans-Bel-Fev-Nica} (which extends Definition~4.4.1 in \cite{SpeHab}), we see that $\mu$ is the $\mc{B}$-valued free Poisson distribution with parameters $(\lambda, \alpha)$.
\end{proof}

See Corollary~\ref{Cor:Poisson} for a follow-up.

\begin{Remark}
For general (not necessarily self-adjoint) $\set{\lambda_i}$ and general (not necessarily positive) $\set{\alpha_i}$, we may still define
\[
\mu = J
\begin{pmatrix}
\lambda_1, & \lambda_2, & \lambda_3, & \lambda_4, & \ldots \\
\alpha_1, & \alpha_2, & \alpha_3, & \alpha_4, & \ldots
\end{pmatrix}.
\]
via the combinatorial formula in Proposition~\ref{basejacobi}. This $\mu$ is now only an algebraic non-commutative distribution. Then numerous results above still hold. We may also define $\mc{B}$-valued semicircular, free Poisson etc. distributions with such more general Jacobi parameters.
\end{Remark}

\begin{Remark}
\label{Remark:Phi-B}
The following objects were defined and studied in Section~6 of \cite{Ans-Bel-Fev-Nica}. For any linear map $\alpha : \mc{B} \rightarrow \mc{B}$, one defined a transformation $\mf{B}_\alpha: \Sigma_{alg}(\mc{B}) \rightarrow \Sigma_{alg}(\mc{B})$, which satisfies
\begin{equation}
\label{B_eta}
\left( \mf{B}_\alpha[\mu] \right)^{\uplus (I + \alpha)} = \mu^{\boxplus (I + \alpha)}.
\end{equation}
For such $\alpha$ and a self-adjoint $\lambda \in \mc{B}$, we can define an (algebraic, not necessarily positive) semicircular distribution $\gamma_{\lambda, \alpha}$ with mean $\lambda$ and variance $\alpha$. Then for a certain transformation $\Phi : \Sigma_{alg}(\mc{B}) \rightarrow \Sigma_{alg}(\mc{B})$ also defined there, and any algebraic distribution $\mu$,
\begin{equation}
\label{copy}
\mf{B}_\eta[\Phi[\mu]] = \Phi[\mu \boxplus \gamma_{0, \eta}].
\end{equation}
We will not need the precise definition of $\Phi$ (see Definition~6.8 of \cite{Ans-Bel-Fev-Nica}), but only the following property.
\end{Remark}

\begin{Cor}
\label{Cor:Phi}
If $\mu$ is an algebraic non-commutative distribution with
\[
\mu = J
\begin{pmatrix}
\lambda_1, & \lambda_2, & \lambda_3, & \lambda_4, & \ldots \\
\alpha_1, & \alpha_2, & \alpha_3, & \alpha_4, & \ldots
\end{pmatrix}.
\]
then
\[
\Phi[\mu] = J
\begin{pmatrix}
0, & \lambda_1, & \lambda_2, & \lambda_3, & \lambda_4, & \ldots \\
I, & \alpha_1, & \alpha_2, & \alpha_3, & \alpha_4, & \ldots
\end{pmatrix}.
\]
\end{Cor}

\begin{proof}
By Corollary~7.11 of \cite{Ans-Bel-Fev-Nica},
\[
\Phi[\nu] = \mu_{(0, \nu)}.
\]
So the result follows from Proposition~\ref{Prop:Shift}.
\end{proof}

Scalar-valued free Meixner distributions were defined in \cite{AnsMeixner} and have been extensively studied since. They are, in a certain precise sense, free analogs of the classical Meixner class, which contains most of the explicit distributions encountered in probability theory.

\begin{Example}
Let $\lambda \in \mc{B}$ be self-adjoint, $\eta \in \mc{CP}(\mc{B})$, and $\alpha : \mc{B} \rightarrow \mc{B}$ a linear map such that $\eta + \alpha \in \mc{CP}(\mc{B})$. A (centered) \emph{free Meixner distribution with parameters $(\lambda, \alpha; \eta)$} is the distribution
\begin{equation}
\label{Free-Meixner}
\mathrm{fM}(\lambda, \alpha; \eta) = J
\begin{pmatrix}
0, & \lambda, & \lambda, & \lambda, & \ldots \\
\eta, & \eta + \alpha, & \eta + \alpha, & \eta + \alpha, & \ldots
\end{pmatrix}.
\end{equation}
Note that $\mathrm{fM}(0, 0; \eta)$ are the semicircular distributions; $\mathrm{fM}(\lambda, 0; \eta)$ the free Poisson distributions; and (as discussed in Remark~\ref{Remark:Bernoulli}) $\mathrm{fM}(\lambda, -\eta; \eta)$ the Bernoulli distributions. In particular, $\alpha$ is not assumed to itself be positive.
\end{Example}

\begin{Prop}
\label{Prop:Meixner}
For fixed $\lambda, \alpha$, free Meixner distributions form a free convolution semigroup with respect to parameter $\eta$: whenever $\alpha + \eta_1, \alpha + \eta_2 \in \mc{CP}(\mc{B})$,
\[
\mathrm{fM}(\lambda, \alpha; \eta_1) \boxplus \mathrm{fM}(\lambda, \alpha; \eta_2) = \mathrm{fM}(\lambda, \alpha; \eta_1 + \eta_2)
\]
and if $I + \alpha \in \mc{CP}(\mc{B})$, then $\mathrm{fM}(\lambda, \alpha; \eta) = \mathrm{fM}(\lambda, \alpha; I)^{\boxplus \eta}$. It also follows that for such $\alpha$ and the transformation $\mf{B}_\eta$,
\[
\mf{B}_\eta[\mathrm{fM}(\lambda, \alpha; I)] = \mathrm{fM}(\lambda, \eta + \alpha; I)
\]
\end{Prop}

\begin{proof}
Let $\mu = \mathrm{fM}(\lambda, \alpha; I)$ be defined via equation \eqref{Free-Meixner}; since we are not assuming that $I + \alpha \in \mc{CP}(\mc{B})$, we may only conclude that $\mu \in \Sigma_{alg}(\mc{B})$. Nevertheless, Corollary~\ref{Cor:Phi} applies, and states that $\mu = \Phi[\gamma_{\lambda, I + \alpha}]$. Now applying identities~\eqref{B_eta} and \eqref{copy}, and using the free convolution property of semicircular distributions,
\[
\Phi[\gamma_{\lambda, I + \alpha}]^{\boxplus \eta}
= \left( \mf{B}_{\eta - I}[\Phi[\gamma_{\lambda, I + \alpha}]] \right)^{\uplus \eta}
= \left( \Phi[\gamma_{\lambda, I + \alpha} \boxplus \gamma_{0, \eta - I}] \right)^{\uplus \eta}
= \left( \Phi[\gamma_{\lambda, \alpha + \eta}] \right)^{\uplus \eta}.
\]
In other words, using also Proposition~\ref{Prop:Boolean},
\[
\mu^{\boxplus \eta} = J
\begin{pmatrix}
0, & \lambda, & \lambda, & \lambda, & \ldots \\
\eta, & \eta + \alpha, & \eta + \alpha, & \eta + \alpha, & \ldots
\end{pmatrix}
\]
i.e. $\mu^{\boxplus \eta} = \mathrm{fM}(\lambda, \alpha; \eta)$. The semigroup property follows. For the final statement, we again observe that
\[
\mf{B}_\eta[\mathrm{fM}(\lambda, \alpha; I)]
= \mf{B}_{\eta}[\Phi[\gamma_{\lambda, I + \alpha}]]
= \Phi[\gamma_{\lambda, I + \alpha + \eta}]
= \mathrm{fM}(\lambda, \eta + \alpha; I) \qedhere
\]
\end{proof}

In the scalar-valued case, the following proposition says that $R(z)$ satisfies a quadratic equation, a well-known result, see Theorem~3(c) in \cite{AnsFree-Meixner}.

\begin{Prop}
If $\mu$ is a free normalized Meixner distribution $\mathrm{fM}(\lambda, \alpha; I)$, then
\[
b^{-1} R_\mu(b) b^{-1} = 1_{\mc{B}} + \lambda R_\mu(b) b^{-1} + \alpha[R_\mu(b) b^{-1}] R_\mu(b) b^{-1}.
\]
\end{Prop}

\begin{proof}
If
\[
\mu = J
\begin{pmatrix}
0, & \lambda, & \lambda, & \lambda, & \ldots \\
I, & I + \alpha, & I + \alpha, & I + \alpha, & \ldots
\end{pmatrix},
\]
then by Proposition~\ref{Prop:Shift},
\begin{equation*}
M_\mu(b) = \left( 1_{\mc{B}} - b M_\nu(b) b \right)^{-1}
\end{equation*}
and
\[
M_\mu(b)^{-1} M_\nu(b) = M_\nu(b) - b M_\nu(b) b M_\nu(b),
\]
where $\nu$ is a semicircular distribution with mean $\lambda$ and covariance $I + \alpha$. By the same proposition,
\[
M_\nu(b) = 1_{\mc{B}} + \lambda b M_\nu(b) + (I + \alpha)[b M_\nu(b)] b M_\nu(b),
\]
thus
\[
M_\mu(b)^{-1} M_\nu(b) = 1_{\mc{B}} + \lambda b M_\nu(b) + \alpha[b M_\nu(b)] b M_\nu(b).
\]
Now using
\[
b M_\nu(b) = (1_{\mc{B}} - M_\mu(b)^{-1}) b^{-1} = (M_\mu(b) - 1_{\mc{B}}) (b M_\mu(b))^{-1},
\]
we get
\[
\begin{split}
& M_\mu(b)^{-1} b^{-1} (1_{\mc{B}} - M_\mu(b)^{-1}) b^{-1} \\
&\qquad = 1_{\mc{B}} + \lambda (1_{\mc{B}} - M_\mu(b)^{-1}) b^{-1} + \alpha[(1_{\mc{B}} - M_\mu(b)^{-1}) b^{-1}] (1_{\mc{B}} - M_\mu(b)^{-1}) b^{-1},
\end{split}
\]
or equivalently
\begin{multline*}
b M_\mu(b))^{-1} (M_\mu(b) - 1_{\mc{B}}) (b M_\mu(b))^{-1} \\
= 1_{\mc{B}} + \lambda (M_\mu(b) - 1_{\mc{B}}) (b M_\mu(b))^{-1} + \alpha[(M_\mu(b) - 1_{\mc{B}}) (b M_\mu(b))^{-1}] (M_\mu(b) - 1_{\mc{B}}) (b M_\mu(b))^{-1}.
\end{multline*}
Using the definition \eqref{Free-cumulant-GF} of the free cumulant generating function,
\[
b^{-1} R_\mu(b) b^{-1} = 1_{\mc{B}} + \lambda R_\mu(b) b^{-1} + \alpha[R_\mu(b) b^{-1}] R_\mu(b) b^{-1}. \qedhere
\]
\end{proof}

\begin{Remark}
\label{Remark:Bernoulli}
Proposition~\ref{Prop:Meixner} is most interesting in the somewhat subtle case of (centered) Bernoulli distributions, which according to Example~\ref{Ex:Bernoulli} are
\[
\mu = J
\begin{pmatrix}
0, & \lambda, & 0, & 0, & \ldots \\
\alpha, & 0, & 0, & 0, & \ldots
\end{pmatrix}.
\]
Since the Jacobi parameter $\alpha_2 = 0$, the values of $\lambda_i, i \geq 3$ can in fact be defined arbitrarily and still give the same distribution. Exactly one choice will make $\mu$ a free Meixner distribution, namely
\[
\mu = J
\begin{pmatrix}
0, & \lambda, & \lambda, & \lambda, & \ldots \\
\alpha, & 0, & 0, & 0, & \ldots
\end{pmatrix}
\]
so that $\mu = \mathrm{fM}(\lambda, -\alpha; \alpha)$.
\end{Remark}

\begin{Cor}
\label{Cor:Poisson}
Let $\mu_N$ be a Bernoulli distribution with Jacobi parameters
\[
\mu_N = J
\begin{pmatrix}
\frac{1}{N} \lambda_1 + o(\frac{1}{N}), & \frac{1}{N} \lambda_1 + \lambda + o(1), & 0, & 0, & \ldots \\
\frac{1}{N} \alpha + o(\frac{1}{N}), & 0, & 0, & 0, & \ldots
\end{pmatrix}.
\]
Then $\mu_N^{\boxplus N} \rightarrow \nu$, where $\nu$ is a free Poisson distribution with mean $\lambda_1$ and parameters $(\lambda, \alpha)$,
\[
\nu = J
\begin{pmatrix}
\lambda_1, & \lambda_1 + \lambda, & \lambda_1 + \lambda, & \lambda_1 + \lambda, & \ldots \\
\alpha, & \alpha, & \alpha, & \alpha, & \ldots
\end{pmatrix}.
\]
\end{Cor}

\begin{proof}
We first use Proposition~\ref{Prop:Mean} and Remark~\ref{Remark:Bernoulli} to write
\[
\mu_N = \delta_{\frac{1}{N} \lambda_1 + o(\frac{1}{N})} \boxplus \bar{\mu}_N,
\]
where
\[
\begin{split}
\bar{\mu}_N
& = J
\begin{pmatrix}
0, & \lambda + o(1), & 0, & 0, & \ldots \\
\frac{1}{N} \alpha + o(\frac{1}{N}), & 0, & 0, & 0, & \ldots.
\end{pmatrix} \\
& = J
\begin{pmatrix}
0, & \lambda + o(1), & \lambda + o(1), & \ldots \\
\frac{1}{N} \alpha + o(\frac{1}{N}), & 0, & 0, & \ldots.
\end{pmatrix}.
\end{split}
\]
Then
\[
\mu_N^{\boxplus N} = \delta_{\lambda_1 + o(1)} \boxplus \bar{\mu}_N^{\boxplus N}
\]
and, applying Proposition~\ref{Prop:Meixner},
\[
\bar{\mu}_N^{\boxplus N} = J
\begin{pmatrix}
0, & \lambda + o(1), & \lambda + o(1), & \ldots \\
\alpha + N o(\frac{1}{N}), & \frac{N-1}{N} \alpha + (N-1) o(\frac{1}{N}), & \frac{N-1}{N} \alpha + (N-1) o(\frac{1}{N}), & \ldots.
\end{pmatrix}.
\]
We conclude that $\mu_N^{\boxplus N} \rightarrow \nu$.
\end{proof}

Note that Theorem~4.4.3 in \cite{SpeHab} proves the usual compound Poisson limit theorem, which implies a particular case of the above with $\lambda_1 = t a$, $\lambda_2 = a$  and $\alpha[b] = t a b a$.

\begin{Example}
If
\[
\mu = J
\begin{pmatrix}
0, & 0, & 0, & 0, & \ldots \\
2 \alpha, & \alpha, & \alpha, & \alpha, & \ldots
\end{pmatrix},
\]
in other words $\mu = \mathrm{fM}(0, - \alpha; 2 \alpha)$, then it is natural to call $\mu$ the $\mc{B}$-valued arcsine distribution. Indeed, recall that in the scalar setting, the arcsine law is the Boolean convolution square of a semicircular distribution, and also free convolution square of a Bernoulli distribution. For $\mu$ as above, the Boolean cumulant generating function from equation~\eqref{Boolean} is
\[
B_\mu(b) = 2 \alpha[b M_\nu(b)] b,
\]
where $\nu$ is the centered semicircular distribution with variance $\alpha$. Since for the semicircular distribution, by \eqref{Semicircular1}, $B_\nu(b) = \alpha[b M_\nu(b)] b$, it follows that
\[
\mu = \nu^{\uplus 2}.
\]
On the other hand, for the centered Bernoulli distribution $\rho$ with covariance $\alpha$, $\rho = \mathrm{fM}(0, -\alpha; \alpha)$. So by Proposition~\ref{Prop:Meixner}
\[
\rho^{\boxplus 2} = \mathrm{fM}(0, - \alpha; 2 \alpha) = \mu.
\]
It follows from Theorem~3.2 in \cite{Belinschi-Popa-Vinnikov-Semicircle} that in the case when $\alpha[b] = a b a$ for some $a$, this arcsine law is the same as in that paper, and in particular appears as the limit law in the monotone central limit theorem.
\end{Example}

\begin{Example}
If $\alpha, \eta - \alpha \in \mc{CP}(\mc{B})$, it is natural to call
\[
\mu = J
\begin{pmatrix}
0, & 0, & 0, & 0, & \ldots \\
\eta, & \eta - \alpha, & \eta - \alpha, & \eta - \alpha, & \ldots
\end{pmatrix},
\]
the $\mc{B}$-valued free binomial distributions, since in the particular case when $\eta = \tau \circ \alpha$,
\[
\mu = \mathrm{fM}(0, - \alpha; \tau \circ \alpha) = \mathrm{fM}(0, - \alpha; \alpha)^{\boxplus \tau}
\]
is a free convolution power of a $\mc{B}$-valued Bernoulli distribution.
\end{Example}

The following proposition computes explicitly the moments of free binomial distributions, arising as free convolution powers of (the distribution of) a special operator $a$. See Section \ref{section5} for a concrete example of such $a$.

\begin{Prop}
\label{Bernoulli}
Let $(\mc{A}, \mc{B}, E)$ be a non-commutative probability space. Let $a \in \mc{A}$ be such that $E[a] = 0$ and $a \mc{B} a \subset \mc{B}$, and denote $\alpha(b) = a b a$. Then
\begin{enumerate}
\item
$a$ has a Bernoulli distribution with parameter $\alpha$.
\item
\label{Bernoulli2}
Taking all $\lambda_i = 0$ and $\alpha_i = \alpha$, for even $n$, $T_\pi(b_0 X b_1 \ldots X b_n) = b_0 a b_1 \ldots a b_n$, and in particular does not depend on $\pi$.
\item
For $t \geq 1$, the odd moments of $\mu_a^{\boxplus t}$ are zero, and the even ones are
\[
\mu_a^{\boxplus t}[b_0 X \ldots X b_n]
= m_n(t) b_0 a b_1 \ldots a b_n.
\]
Here $m_n(2) = \binom{2n}{n}$, and in general for $n > 0$,
\[
m_n(t) = t^{2n} - \frac{t}{2} \sum_{k=1}^n \frac{1}{2k-1} \binom{2k}{k} (t-1)^k t^{2(n-k)}.
\]
\end{enumerate}
\end{Prop}

\begin{proof}
For part (a), we verify that for $n$ even,
\[
E[b_0 a b_1 \ldots a b_n] = b_0 \alpha(b_1) b_2 \alpha(b_3) \ldots \alpha(b_{n-1}) b_n
\]
and for $n$ odd,
\[
E[b_0 a b_1 \ldots a b_n] = b_0 \alpha(b_1) b_2 \alpha(b_3) \ldots \alpha(b_{n-2}) b_{n-1} E[a] b_n = 0.
\]
For part (b), it suffices to note that
\[
\alpha(b_1 \alpha(b_2) b_3 \alpha(b_4) \ldots \alpha(b_{n-2}) b_{n-1})
= a b_1 a b_2 \ldots a b_{n-1} a
= \alpha(b_1) b_2 \alpha(b_3) b_4 \ldots b_{n-2} \alpha(b_{n-1}).
\]
For part (c), we note that
\[
\mu_a^{\boxplus t} = J
\begin{pmatrix}
0, & 0, & 0, & 0, & \ldots \\
t \alpha, & (t-1) \alpha, & (t-1) \alpha, & (t-1) \alpha, & \ldots
\end{pmatrix},
\]
and so for $n$ even
\[
\begin{split}
\mu_a^{\boxplus t}[b_0 X \ldots X b_n]
& = \sum_{\pi \in \NC_{2}(n)} T_\pi(b_0 X b_1 \ldots X b_n) \\
& = \sum_{\pi \in \NC_{2}(n)} t^{\abs{\Outer(\pi)}} (t-1)^{\abs{\Inner(\pi)}} b_0 a b_1 \ldots a b_n
= m_n(t) b_0 a b_1 \ldots a b_n.
\end{split}
\]
Here $\NC_2(n)$ are non-crossing pair partitions, $\Outer(\pi)$, $\Inner(\pi)$ are the outer, respectively, inner blocks of $\pi$, and $m_n(t)$ is the $n$'th moment of the scalar-valued free binomial distribution with parameter $t$. If $t=2$, it is the arcsine distribution, and
\[
m_n(2) = \binom{2n}{n}.
\]
In general, the moment generating function of this scalar-valued distribution is
\[
\sum_{n=0}^\infty m_n(t) z^n = \frac{t - 2 - t \sqrt{1 - 4 (t-1) z^2}}{2 (t^2 z^2 - 1)},
\]
from which the formula for the moments is easily deduced.
\end{proof}

\section{Joint Distributions of $\mathcal{B}$-free Jacobi-Szeg\H{o} distributions.}\label{section3}

Call a \emph{Jacobi-Szeg\H{o} random variable} a $\mc{B}$-valued random variable with a Jacobi-Szeg\H{o} distribution.
In this section, we provide a combinatorial description of the joint moments of $\mc{B}$-free Jacobi-Szeg\H{o} random variables.

\begin{Remark}
Define the set $\mathcal{TCNC}_{1,2}(n)$ to be
set of non-crossing partitions of the set $\{ 1,2, \ldots , n \}$ into blocks of size at most $2$ and where each of the blocks is also assigned one of two colors (red and blue, respectively). For each $i \in \{ 1,2, \ldots , n \}$, its color according to $\pi$ is the color of block of $\pi$ to which is belongs. We define the subset $\mathcal{TCNC}_{2}(n) \subset \mathcal{TCNC}_{1,2}(n)$
as those partitions with pairings and no singletons (this only works for $n$ even).

Setting notation, let $\pi \in \mathcal{NC}_{1,2}(n)$, $P(X) \in \mathcal{B}\langle X \rangle$ a monomial of degree $n$, and let $\mu$ denote a Jacobi-Szeg\H{o} distribution with parameters $\set{(\lambda_i)_{i=0}^\infty, (\alpha_{i})_{i=1}^{\infty}}$. The definition of $T_{\pi}(P(X))$, the moment associated to the partition $\pi$, was given in Proposition~\ref{basejacobi}. Next, let $\pi \in \mathcal{TCNC}_{1,2}(n)$.  Consider two Jacobi-Szeg\H{o} distributions with parameters $\set{(\lambda_i)_{i=1}^\infty, (\alpha_{i})_{i=1}^{\infty}}$ and
$\set{(\tau_i)_{i=1}^\infty, (\beta_{i})_{i=1}^{\infty}}$.  Then $E_{\pi}$ is the moment calculated according to this partition where blue blocks are associated to the first sequence of Jacobi parameters and red the second.  Thus, if
\[
\pi = \{ (1,8), (2,6), (4)\}_{b} \cup \{(3,5), (7) \}_{r}
\]
and
\[
\pi' = \{(1,8), (7) \}_{b} \cup \{(2,6), (3,5), (4)\}_{r}
\]
(the $b$ and the $r$ assign the color), then

\[
E_{\pi} (Xb_{1}Xb_{2}Xb_{3}Xb_{4}Xb_{5}X b_6 X b_7 X) = \alpha_1(b_{1} \alpha_2(b_{2} \beta_1(b_{3} \lambda_1 b_{4}) b_7) \tau_0 b_8)
\]
while
\[
E_{\pi'} (Xb_{1}Xb_{2}Xb_{3}Xb_{4}Xb_{5}X b_6 X b_7 X) = \alpha_1(b_{1} \beta_1(b_{2} \beta_2(b_{3} \tau_2 b_{4})b_{5}) b_6 \lambda_2 b_7).
\]
Crucially, nesting inside a pair of the opposite color implies that the algorithm for applying the automorphisms resets itself (that is, with partition $\pi$, $\beta_{1}$ is applied to $b_{3}$ as opposed to $\beta_{2}$).

More specifically, consider an element $k \in \{ 1, 2, \ldots , n \}$ and $\pi \in \mathcal{TCNC}_{1,2}(n)$.
Assume that $k$ is assigned a blue coloring by the partition $\pi$.
Assume that there exists a red pairing $(c , d)$ which is the red covering of $k$, in the sense that $c < k < d$ and if $(c' , d')$ is another red pairing satisfying the inequality $c' < k < d'$, then $c' < c < d < d'$.  The depth of $k$ is equal to $\ell +1 $ where $(a_{1} , b_{1}) , (a_{2} , b_{2}), \ldots , (a_{\ell} , b_{\ell})$ are a maximal collection of blue pairs in $\pi$ such that $$ c < a_{1} < a_{2} < \cdots < a_{\ell} < k < b_{\ell} < \cdots < b_{2} < b_{1} < d$$
with the convention that $k$ is depth $1$ if no such blue pairings exist.
If there exists no red pair $(c,d)$ in $\pi$ such that $c < k < d$ then the depth of $k$ is simply the number $\ell + 1$ where the blue pairs
$\{ (a_{i} , b_{i}) \}_{i = 1}^{\ell}$ are a maximal family satisfying
$$ a_{1} < a_{2} < \cdots < a_{\ell} < k < b_{\ell} < \cdots < b_{2} < b_{1} $$
with the convention that $k$ is depth $1$ if no such blue pairings exist.
If the number $k$ of depth $\ell + 1$ belongs to a pair $(a,b)$ in $\pi$ then this pair produces the automorphism $\alpha_{\ell + 1}$.
If the number $k$ belongs to a singleton $\{ k \}$ in $\pi$ then this singleton produces the element $\lambda_{\ell}$.
Note that in the special case that $\pi$ is only one color, this algorithm simply reduces to Proposition~\ref{basejacobi}.
\end{Remark}

\begin{Example}
Consider two Jacobi-Szeg\H{o} random variables $X_{1}$ and $X_{2}$
with Jacobi parameters $\set{(0)_{i=1}^\infty, (\alpha_{i})_{i = 1}^{\infty}}$ and  $\set{(0)_{i=1}^\infty, (\beta_{i})_{i = 1}^{\infty}}$ (since all their odd moments are zero, by analogy with the scalar case we could call them \emph{symmetric} Jacobi-Szeg\H{o} random variables).
Consider the expectation
\[
E(X_{1}b_{1}X_{1}b_{2}X_{2}b_{3}X_{2}b_{4}X_{1}b_{5}X_{1}b_{6}X_{2}b_{7}X_{2}).
\]
According to Theorem \ref{maintheorem} below, this moment should be equal to a sum of terms
\begin{equation}\label{demonstration}
E_{\pi}(X_{1}b_{1}X_{1}b_{2}X_{2}b_{3}X_{2}b_{4}X_{1}b_{5}X_{1}b_{6}X_{2}b_{7}X_{2})
\end{equation}
 where $\pi \in \mathcal{TCNC}_{1,2}(4)$ is such that the blue colors are assigned to the $X_{1}$'s and the red to the
$X_{2}$'s.  These correspond to the partitions in Figure~\ref{Three-pairs}, labeled A, B and C from left to right:

\begin{figure}[hhh]
\label{Three-pairs}
\includegraphics[width=\textwidth]{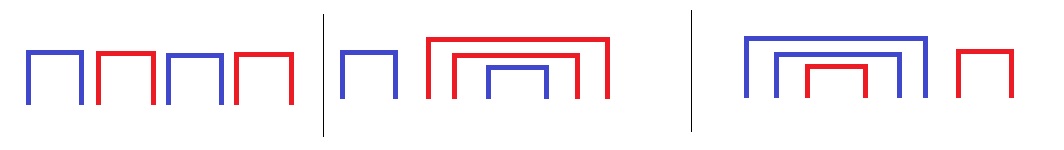}
\caption{Partitions A, B, and C.}
\end{figure}

which produces the  expectation
\[
\begin{split}
\alpha_{1}(b_{1})b_{2}\beta_{1}(b_{3})b_{4}\alpha_{1}(b_{5})b_{6}\beta_{1}(b_{7})
& +  \alpha_{1}(b_{1})b_{2}\beta_{1}(b_{3} \beta_{2}( b_{4}\alpha_{1}(b_{5}) b_{6}    )b_{7}) \\
& +  \alpha_{1}(b_{1} \alpha_{2}(b_{2} \beta_{1}(b_{3})b_{4}) b_{5})b_{6}   \beta_{1}(b_{7}).
\end{split}
\]
\end{Example}

\begin{theorem}\label{maintheorem}
Let $X_{1}$ and $X_{2}$ denote $\mathcal{B}$-free Jacobi-Szeg\H{o} random variables.
Then the joint moment $$E(X_{\epsilon_{1}} b_{1} X_{\epsilon_{2}}\cdots b_{d-1} X_{\epsilon_{d}}) $$ is equal to the sum of the terms
$$ E_{\pi}( X_{\epsilon_{1}} b_{1} X_{\epsilon_{2}}\cdots b_{d-1} X_{\epsilon_{d}} )$$
where $\pi \in \mathcal{TCNC}_{1,2}(d)$, $\epsilon_{i} \in \{1,2\}$ for $i = 1 , 2 , \ldots , d$ and the partition $\pi$ is such that all blue blocks consist of $X_{1}$'s and all red blocks consist of $X_{2}$'s.
\end{theorem}
\begin{proof}
Setting notation, we  consider a family of monomials $P_{i}(X) \in \mathcal{B}\langle X \rangle$ for $i = 1 , \ldots , n$. We prove our theorem for $$ E(P_{1}(X_{i_{1}}) P_{2}(X_{i_{2}}) \cdots P_{n}(X_{i_{n}})) $$
where  $i_{j} \in \{ 1, 2\}$ and $i_{1} \neq i_{2}, i_{2} \neq i_{3} , \ldots , i_{n-1} \neq i_{n}$.  We refer to  the monomial $ P_{j}(X_{i_{j}})$ as the \textit{$j$th interval}. Let $d_{j} = deg(P_{j})$, so that $d = \sum_{j=1}^{n} d_{j}$.
We say that $\pi \in \mathcal{TCNC}_{1,2}(d)$ \textit{fixes the $j$th interval} if the elements in the interval
$$I_j = [d_{1} + d_{2} + \cdots + d_{j-1} + 1, d_{1} + d_{2} + \cdots + d_{j-1} + d_{j}]$$ are singletons or paired with other elements from the same interval.

Proceeding by induction on the number of intervals, the case $n=1$ is simply the computational algorithm for the moments of Jacobi-Szeg\H{o} random variables in Proposition~\ref{basejacobi} since only one color will be permitted and $\mathcal{TCNC}_{1,2}(d)$ will therefore collapse to $\mathcal{NC}_{1,2}(\deg(P_{1}))$.  Thus, we assume that the theorem holds for any monomial with less than $n$ intervals.

Denoting $\overline{P_{j}(X_{i_{j}})}=  P_{j}(X_{i_{j}}) - E(P_{j}(X_{i_{j}}))$, since
\[
E(\overline{P_{1}(X_{i_{1}})} \overline{P_{2}(X_{i_{2}})} \ldots \overline{P_{n}(X_{i_{n}})}) = 0
\]
by freeness,
\begin{equation}
\label{Expansion}
E(P_{1}(X_{i_{1}}) P_{2}(X_{i_{2}}) \cdots P_{n}(X_{i_{n}}))
= \sum_{k=1}^n (-1)^{k+1} \sum_{\substack{S \subset \set{1, 2, \ldots, n} \\ \abs{S} = k}} E(A_1(S) A_2(S) \ldots A_n(S)),
\end{equation}
where
\[
A_j(S) =
\begin{cases}
E(P_{j}(X_{i_{j}})), & j \in S, \\
P_{j}(X_{i_{j}}), & j \not \in S.
\end{cases}
\]
Since $k \geq 1$, and each $E(P_{j}(X_{i_{j}})) \in \mc{B}$, each of the words $A_1(S) A_2(S) \ldots A_n(S)$ contains less than $n$ intervals of $X_1$'s and $X_2$'s. So we may apply the induction hypothesis to each of these words, and write each
\[
E(A_1(S) A_2(S) \ldots A_n(S)) = \sum_{\pi \in \mc{TCNC}_{1,2}\left( d - \sum_{j \in S} d_j \right)} E_\pi (A_1(S) A_2(S) \ldots A_n(S)).
\]
as the sum of the appropriate $E_\pi$. Let $V \subset S$ consist of those $j$ for which $E(P_{j}(X_{i_{j}}))$ is covered by a block of $\pi$ of the same color as $X_{i_j}$. For each $j \in S \setminus V$, we may use Proposition~\ref{basejacobi} to replace each $E(P_{j}(X_{i_{j}}))$ with the sum
\begin{equation}
\label{Replacement}
E(P_{j}(X_{i_{j}})) = \sum_{\sigma_j \in \NC(d_j)} T_{\sigma_j}(P_{j}(X_{i_{j}})),
\end{equation}
Now combine $\pi$ and $\sigma_j, j \in S \setminus V$ into a single partition
\[
\tau = \pi \cup \bigcup_{j \in S \setminus V }\sigma_j \in \mc{TCNC}_{1,2} \left(d - \sum_{j \in V} d_j \right)
\]
of $\bigcup_{j \not \in V} I_j$. Since for $j \in S \setminus V$, $E(P_{j}(X_{i_{j}}))$ is either not covered by $\pi$ or is covered by a block of $\pi$ of the opposite color from $X_{i_j}$, because of the way $E_\pi$ is defined,
\begin{equation}
\label{Subs}
E_\pi (B_1(S, V) B_2(S, V) \ldots B_n(S, V))
= E_{\pi \cup \bigcup_{j \in S \setminus V} \sigma_j} (A_1(V) A_2(V) \ldots A_n(V)),
\end{equation}
where
\[
B_j(S, V) =
\begin{cases}
E(P_{j}(X_{i_{j}})), & j \in V, \\
T_{\sigma_j}(P_{j}(X_{i_{j}})), & j \in S \setminus V, \\
P_{j}(X_{i_{j}}), & j \not \in S.
\end{cases}
\]
Note that \eqref{Subs} will \emph{not} in general hold if we use the substitution \eqref{Replacement} for any $j \in V$. Plugging \eqref{Subs} into \eqref{Expansion}, we obtain
\[
\begin{split}
& E(P_{1}(X_{i_{1}}) P_{2}(X_{i_{2}}) \cdots P_{n}(X_{i_{n}})) \\
& = \sum_{k=1}^n (-1)^{k+1} \sum_{\substack{S \subset \set{1, 2, \ldots, n} \\ \abs{S} = k}} \sum_{\pi \in \mc{TCNC}_{1,2}\left( d - \sum_{j \in S} d_j \right)} \sum_{\substack{\sigma_j \in \NC(d_j) \\ j \in S \setminus V}} E_{\pi \cup \bigcup_{j \in S \setminus V} \sigma_j} (A_1(V) A_2(V) \ldots A_n(V)).
\end{split}
\]
Each partition $\tau$ can be represented as $\pi \cup \bigcup_{j \in S \setminus V} \sigma_j$ in several ways. To account for this redundancy, fix $V \subset \set{1, \ldots, n}$, and let $\tau \in \mc{TCNC}_{1,2} \left(d - \sum_{j \in V} d_j \right)$. Assume that it fixes the intervals with indices exactly in the set $U(\tau, V) \subset \set{1,\ldots n} \setminus V$. For each of these intervals $I_j$, one can choose whether the restriction $\tau |_I$ comes from $\sigma_j$ (so that $j \in S \setminus V$) or directly from $\pi$ (and so $j \in U(\tau, V) \setminus S$). Thus the expansion above can be reorganized as
\[
= \sum_{V \subset \set{1, \ldots, n}} \sum_{\tau \in \mc{TCNC}_{1,2} \left(d - \sum_{j \in V} d_j \right)} \sum_{\substack{S: \abs{S} \geq 1 \\ V \subset S \subset V \cup U(\tau,V)}} (-1)^{\abs{S} + 1} E_\tau (A_1(V) A_2(V) \ldots A_n(V)).
\]
The sum over $S$ reduces to
\[
\sum_{\substack{S: \abs{S} \geq 1 \\ V \subset S \subset V \cup U(\tau,V)}} (-1)^{\abs{S} + 1}
= (-1)^{\abs{V} + 1} \sum_{k = \min(1 - \abs{V}, 0)}^{\abs{U(\tau, V)}} (-1)^k \binom{\abs{U(\tau, V)}}{k}
=
\begin{cases}
0, & \abs{V} \geq 1, \\
1, & \abs{V} = 0.
\end{cases}
\]
Thus the only terms which contribute to the sum are those for $V = \emptyset$. Noting that $A_j(\emptyset) = P_{j}(X_{i_{j}})$, we conclude that
\[
E(P_{1}(X_{i_{1}}) P_{2}(X_{i_{2}}) \cdots P_{n}(X_{i_{n}}))
= \sum_{\tau \in \mc{TCNC}_{1,2} \left(d \right)} E_\tau(P_1(X_{i_{1}}) P_{2}(X_{i_{2}}) \cdots P_{n}(X_{i_{n}})). \qedhere
\]
\end{proof}

\section{Analytic Computations}\label{section4}

In this section we will consider sums of $\mc{B}$-free random variables with truncated Jacobi-Szeg\H{o} distributions.  Let
\begin{equation}
\label{Jacobi-parameters}
\mu_{k} = J
\begin{pmatrix}
\lambda_1, & \lambda_2, & \lambda_3, & \lambda_4, & \ldots & \lambda_{k-1}, & \lambda_{k}, & 0, & \ldots  \\
\alpha_{1}, &\alpha_{2}, & \alpha_{3}, & \alpha_{4}, & \ldots & \alpha_{k-1}, & 0 & 0 & \ldots
\end{pmatrix}
\end{equation}
\begin{equation*}
\rho_{\ell} = J
\begin{pmatrix}
\tau_1, & \tau_2, & \tau_3, & \tau_4, & \ldots & \tau_{k-1}, & \tau_{k}, & 0, & \ldots  \\
\beta_{1}, &\beta_{2}, & \beta_{3}, & \beta_{4}, & \ldots & \beta_{\ell-1}, & 0 & 0 & \ldots
\end{pmatrix}
\end{equation*}
We say that $\mu_k$ has depth $k$. We will describe the non-zero moments of $\mu_{k} \boxplus \rho_{\ell}$.

We begin by considering subsets $\mathcal{NC}_{1,2}^{k}(n) \subset \mathcal{NC}_{1,2}(n)$ whose partitions only have pairs of depth less than $k$.
Consider  $\pi \in \mathcal{NC}_{1,2}(n)$. If there exist pairs $\{ (a_{i} , b_{i}) \}_{i=1}^{k} \subset \pi$.
such that $$ a_{1} < a_{2} < \cdots < a_{k} < b_{k} <  \cdots < b_{2} < b_{1}$$
then we have that $\pi \in \mathcal{NC}_{1,2}(n) \setminus \mathcal{NC}_{1,2}^{k}(n)$.

We define subsets of $\mathcal{TCNC}_{1,2}^{k,\ell}(n) \subset \mathcal{TCNC}_{1,2}(n)$ as those elements where the blue pairs have depth less than $k$ and the red pairs have depth less than $\ell$, in the following precise sense.
Let $$ \pi = \{ (a_{i} , b_{i} ) \}_{i=1}^{p} \cup \{ (c_{j} ,d_{j} ) \}_{j=1}^{\ell} \cup \{ e_{m} \}_{m=1}^{n - 2p - 2\ell}$$ where the pairs $(a_{i},b_{i})$ are blue and the pairs $(c_{j} ,d_{j} ) $ are red. $\pi \in \mathcal{TCNC}_{1,2}^{k,\ell}(n)$ if for any indices $i_{1} < i_{2} < \cdots < i_{k}$ such that $$ a_{i_{1}} < a_{i_{2}} < \cdots < a_{i_{k}} < b_{i_{k}} <  \cdots < b_{i_{2}} < b_{i_{1}},$$
there exists a pair $(c_{j},d_{j})$ such that
$$ a_{i_{1}} < c_{j}  < a_{i_{k}} < b_{i_{k}} <  d_{j} < b_{i_{1}}.$$
Moreover, if the red and blue are swapped then the same property must be true with $\ell$ replacing $k$.

We define the subsets $ \mathcal{TCNC}_{2}^{k,\ell}(n) \subset \mathcal{TCNC}_{1,2}^{k,\ell}(n)$ as those $\pi \in \mathcal{TCNC}_{1,2}^{k,\ell}(n)$ with no singletons.
These will be the main focus of the forthcoming computations.
For the readers convenience, the figure above
consists of the $20$ elements in $\mathcal{TCNC}_{2}^{2,2}(6)$.

\begin{figure}
\label{TCNC}
\includegraphics[width=\textwidth]{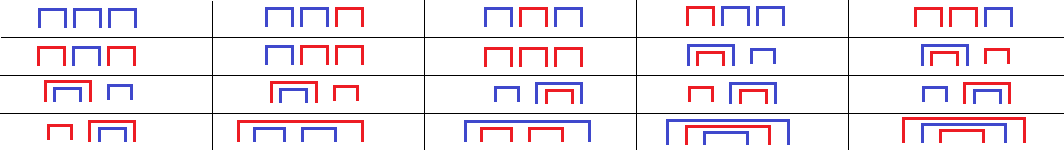}
\caption{$\mathcal{TCNC}^{2,2}_2(6)$}
\end{figure}

We provide a recursive  definition of these sets.
Indeed, $\mathcal{TCNC}_{2}^{k,\ell}(2n)$ is the set of all two-color non-crossing partitions $\pi $ of $\{ 1, 2 , \ldots , 2n\}$ whose coloring respects the pairing with the property that there exists an interval $I \subset \{ 1, 2 , \ldots , 2n\}$ such that
\begin{enumerate}
\item\label{rec1}  The elements of $I$ are blue (resp. red) and it is bordered by red (resp. blue) elements.
\item\label{rec2} $\pi|_{I} \in \mathcal{NC}_{k}(|I|)$ (resp.  $\pi|_{I} \in \mathcal{NC}_{\ell}(|I|))$.
\item\label{rec3} $ \pi|_{\{ 1, 2 , \ldots , 2n\} \setminus I } \in \mathcal{TCNC}_2^{k,\ell}(n - |I|/2)$ (with the obvious shift of the underlying set).
\end{enumerate}

\begin{Lemma}\label{basecase}
Let $X$ be a random variable with distribution $\mu_{k}$ from equation~\eqref{Jacobi-parameters}.
Then $$ E(Xb_{1}X \cdots b_{n-1} X) = \sum_{\pi \in \mathcal{NC}_{1,2}^{k}(n)} T_{\pi}(Xb_{1}X \cdots b_{n-1} X).$$
\end{Lemma}
\begin{proof}
We have that
$$ E(Xb_{1}X \cdots b_{n-1} X) = \sum_{\pi \in \mathcal{NC}_{1,2}(n)} T_{\pi}(Xb_{1}X \cdots b_{n-1} X) .$$
If $\pi \in \mathcal{NC}_{1,2}(n) \setminus \mathcal{NC}_{1,2}^{k}(n)$ then the partition has a pairing of depth of at least $k$.  This implies that
the $k$th completely positive map will be applied.  However, this is the $0$ automorphism so this term vanishes.  Our result follows.
\end{proof}

We have the following corollary to Theorem \ref{maintheorem}.

\begin{Corollary}\label{count0}
Assume that the random variables $X_{1}$ and $X_{2}$ in the statement of Theorem \ref{maintheorem} have distributions of depth $k$ and $\ell$ respectively.
Then the selection of $\pi$ may be restricted to the subset $\mathcal{TCNC}_{1,2}^{k,\ell}(n) \subset \mathcal{TCNC}_{1,2}(n) $.
\end{Corollary}

We now produce a specific example of a convolution of these distributions with $\mathcal{B} = \mathbb{C}$.
This will also provide a convenient method for counting the size of the sets $\mathcal{TCNC}_{2}^{k,k}(n)$ through free probabilistic methods.

Consider a non-commutative probability space $(M_{k}(\mathbb{C}), \phi_{k})$ where $\phi_{k} (X) = e_{1,1}Xe_{1,1}$.
Consider the self adjoint random variable $$X_{k} = \sum_{i=1}^{k-1} e_{i , i+1 } + e_{i+1 , i}.$$
Since $X$ is a finite tridiagonal matrix (with zero diagonal), the next proposition follows immediately from Theorem~\ref{Thm;Favard}.

\begin{Proposition}\label{count1}
We have that $\phi_{k}(X_{k}^{2n}) = |\mathcal{NC}_{k}(2n)|$.
\end{Proposition}

\begin{Remark}
We refer to the probability measure arising  as the distribution of $X_{k}$ with respect to $\phi_{k}$ as $\nu_{k}$. In fact, we can compute $\nu_k$ (and so $|\mathcal{NC}_{k}(2n)|$) explicitly. Namely, from the continued fraction expansion it follows that
\[
G_{\nu_k}(z) = \frac{U_{k-1}(z)}{U_k(z)},
\]
where $U_k(2 \cos \theta) = \frac{\sin(k+1) \theta}{\sin \theta}$ are the Chebyshev polynomials of the second kind (this is a basic fact from the Gaussian quadrature approximation, see \cite{Chihara-book}). The measures can be recovered from this via Stieltjes inversion, while the moments are the coefficient in the asymptotic expansion around infinity. In fact, a short computation shows that
\[
\nu_k = \sum_{j=1}^k a_j \delta_{x_j},
\]
where $x_j = 2 \cos \frac{j}{k+1} \pi$ and $a_j = \frac{1}{k+1} \left( 1 - \cos \frac{2j}{k+1} \pi \right)$.
\end{Remark}

\begin{Corollary}\label{count2}
We have that $|\mathcal{TCNC}_{k,\ell}(2n)| = \nu_{k} \boxplus \nu_{\ell}(t^{2n})$ where $\nu_{k}$ and $\nu_{\ell}$ are the probability measures arising in Proposition \ref{count1}.
\end{Corollary}
\begin{proof}
This is simply a combination of Corollary \ref{count0} and Proposition \ref{count1}.
\end{proof}

We establish the convention that $\nu_{1} = \delta_{0}$, the Dirac mass at $0$.

\begin{Lemma}\label{recurrsiveform}
For $n > 1$, we have that \begin{equation}
G_{\nu_{n}}(z) = \frac{1}{z -G_{\nu_{n - 1}}(z) }.
\end{equation}
\end{Lemma}
\begin{proof}
This follows immediately from the continued fraction expansion in Theorem~\ref{Thm;Favard} and the relation
\[
G_\mu(z) = \frac{1}{z} M_\mu(1/z). \qedhere
\]
\end{proof}

Let $\mu_{n,n} = \nu_{n} \boxplus \nu_{n}$.  The following Corollary will prove useful in computing its distribution.
\begin{Corollary}\label{subordination}
For all $z \in \mathbb{C}^{+}$, we have
$$F_{\mu_{n,n}} (z + G_{\nu_{n-1}}(z)) = z - G_{\nu_{n-1}}(z). $$
\end{Corollary}
\begin{proof}
Recall \cite{VDN} that $F^{\langle -1 \rangle}(z) - z$ (where $\langle -1 \rangle$ denotes the inverse under composition) linearizes free convolution: on an appropriate domain,
\begin{equation}
\label{Linearization}
F_{\nu \boxplus \nu}^{\langle -1 \rangle} (z) - z
=  2 (F_{\nu}^{\langle -1 \rangle} (z) - z).
\end{equation}
So on such a domain, we have that \begin{align*}
& \ \ \ \ F^{\langle -1 \rangle}_{\mu_{n,n}}(z) = 2F_{\nu_{n}}^{\langle -1 \rangle} - z \\
\Rightarrow & \ \ \ \    F^{\langle -1 \rangle}_{\mu_{n,n}}(F_{\nu_{n}}(z))  = 2z - F_{\nu_{n}}(z) \\
\Rightarrow & \ \ \ \   F^{\langle -1 \rangle}_{\mu_{n,n}}( z - G_{\nu_{n-1}}(z) )  = z + G_{\nu_{n-1}}(z)
\end{align*}
where the last implication follows  from Lemma \ref{recurrsiveform}.  Our claim follows on an appropriate domain, and on all of $\mathbb{C}^{+}$ through continuation.
\end{proof}

We set notation before proving the main result of this section.
For $q < p$, we define $\mathcal{P_{O}}(p,q)$ to denote the set of interval partitions of $\{ 1 , 2 , \ldots , p \}$ into $q$ distinct blocks of odd length.
Given an element $\pi \in \mathcal{P_{O}}(p,q)$ we define $\pi_{i} $ to be the $i$th block, in ascending order, for $i = 1 , 2 , \ldots , q$.
We denote by $|\pi_{i}|$ the number of elements in  this interval.

The following theorem shows that the  formula for the Cauchy transform in Corollary \ref{subordination} translates to a recursive formula for the convolved measures $\mu_{k,k}$.  This is proven by stripping the coefficients of the Cauchy transform $G_{\mu_{k,k}}$.

\begin{theorem}\label{momentformula}
Let $M^{(k)}_{n}$ and $m^{(\ell)}_{n}$ denote the $nth$ moments for the measures $\mu_{k,k}$ and $\nu_{\ell}$, respectively.
The measure $\mu_{k,k}$ is symmetric and we have the following recursive formula for the even moments:

\begin{equation}\label{mess1}
 M^{(k)}_{2n} = S_{n,k} - T_{n,k}
\end{equation}
where
\begin{equation}\label{mess2}
 S_{n,k} = 2\sum_{i=n}^{2n -1}\binom{2n-1}{i} \sum_{\pi \in \mathcal{P_{O}}(i, 2n - i)} m_{|\pi_{1}| - 1}^{(k-1)} m_{|\pi_{2}| - 1}^{(k-1)}  \cdots
m_{|\pi_{2n - i}| - 1}^{(k-1)},
\end{equation}
\begin{equation}\label{mess3}
T_{n,k} =  \sum_{j=1}^{n - 2} M_{2j}^{(k)} \sum_{p= (n-j) - 1 }^{2(n-j) - 1}  \binom{2(n-j)-1}{p} \left[ R_{p+1, j ,n,k}  -  R_{p, j ,n,k} \right],
\end{equation}
and
\begin{equation}\label{mess4}
 R_{p, j ,n,k} =  \left( \sum_{\pi \in \mathcal{P_{O}}(p , 2(n-j) - p)}  (m_{|\pi_{1}| - 1}^{(k-1)} m_{|\pi_{2}| - 1}^{(k-1)}  \cdots
m_{|\pi_{2n - i}| - 1}^{(k-1)})  \right)
\end{equation}
\end{theorem}
\begin{proof}
Consider the Cauchy transform \begin{equation}\label{basicCauchy} G_{\mu_{k,k}}(z) = \sum_{p=0}^{\infty} \frac{ M^{(k)}_{2p}}{z^{2p + 1}}.\end{equation}
Rewriting \eqref{basicCauchy} and taking limits, we have
\begin{align}
 M^{(k)}_{2n} &= \lim_{|z| \uparrow \infty} z^{2n + 1} \left[ G_{\mu_{k,k}}(z) - \frac{ M^{(k)}_{0}}{z} + \frac{ M^{(k)}_{2}}{z^{3}} + \cdots + \frac{ M^{(k)}_{2(n-1)}}{z^{2n -1} } \right] \\ \label{Limit1}
&= \lim_{|z| \uparrow \infty} (z + G_{\nu_{k-1}}(z))^{2n + 1} \left[ G_{\mu_{k,k}}((z + G_{\nu_{k-1}}(z))) - \sum_{j=0}^{n-1} \frac{ M^{(k)}_{2j}}{(z + G_{\nu_{k-1}}(z))^{2j+1}} \right] \\ \label{Limit2}
&= \lim_{|z| \uparrow \infty} \left[ \frac{z - G_{\nu_{k-1}}(z)}{z +G_{\nu_{k-1}}(z)} \right] (z + G_{\nu_{k-1}}(z))^{2n+1} \left[ \frac{1}{z - G_{\nu_{k-1}}(z)} - \sum_{j=0}^{n-1} \frac{ M^{(k)}_{2j}}{(z + G_{\nu_{k-1}}(z))^{2j+1}} \right] \\ \label{Limit3}
& = \lim_{|z| \uparrow \infty} \left[ (z + G_{\nu_{k-1}}(z))^{2n} - (z - G_{\nu_{k-1}}(z))(z + G_{\nu_{k-1}}(z))^{2n-1} \right] \\ \nonumber
& + [ M^{(k)}_{2}(z - G_{\nu_{k-1}}(z))(z + G_{\nu_{k-1}}(z))^{2n-3} + M^{(k)}_{4}(z - G_{\nu_{k-1}}(z))(z + G_{\nu_{k-1}}(z))^{2n-5} + \cdots  \\  \nonumber & + M^{(k)}_{2(n-1)}(z - G_{\nu_{k-1}}(z))(z + G_{\nu_{k-1}}(z)) ]
\end{align}
where equality \eqref{Limit1} is justified since $|z + G_{\nu_{k-1}}(z)| \uparrow \infty$ as $|z| \uparrow \infty$.  Equality \eqref{Limit2} is justified since
this is a product of convergent limits and
$$ \lim_{|z| \uparrow \infty} \frac{z - G_{\nu_{k-1}}(z)}{z +G_{\nu_{k-1}}(z)} = 1$$
as well as Corollary \ref{subordination}.
 Since convergence of \eqref{Limit3} is established, we need only identify the constant terms to identify the limit.
We break this into two pieces, letting
$$S(z) = (z + G_{\nu_{k-1}}(z))^{2n} - (z - G_{\nu_{k-1}}(z))(z + G_{\nu_{k-1}}(z))^{2n-1} $$
and setting the remaining terms in \eqref{Limit3} equal to $T(z)$.
We will establish our theorem by showing that the constant term for $S(z)$ is equal to $S_{n,k}$ and the constant term for $T(z)$ is equal to $T_{n,k}$.

We begin with $S(z)$.  Observe that
$$S(z) = 2G_{\nu_{k-1}}(z) (z + G_{\nu_{k-1}}(z))^{2n-1}= 2\sum_{i=0}^{2n - 1} \binom{2n-1}{i} z^{i}  G_{\nu_{k-1}}(z)^{2n - i }   $$
Now, isolating $  z^{i}  G_{\nu_{k-1}}(z)^{2n - i } $, our task devolves to identifying the constant term of this Laurent series.  As $G_{\nu_{k-1}}(z)^{2n - i } = O(z^{i-2n})$, we only receive contributions for  $i \geq n$ so that we focus on
$$ 2\sum_{i=n}^{2n - 1} \binom{2n-1}{i} z^{i}  G_{\nu_{k-1}}(z)^{2n - i }.   $$
Now observe that
$$G_{\nu_{k-1}}(z)^{2n - i } = \left( \sum_{p=0}^{\infty} \frac{m^{(k-1)}_{p}}{z^{p+1} }  \right)^{2n-i}$$
and we must identify the coefficient of the $z^{-i}$ term.  But this is exactly
$$\sum_{\pi \in \mathcal{P_{O}}(i, 2n - i)} m_{|\pi_{1}| - 1}^{(k-1)} m_{|\pi_{2}| - 1}^{(k-1)}  \cdots
m_{|\pi_{2n - i}| - 1}^{(k-1)} $$
since, given $\pi \in \mathcal{P_{O}}(i, 2n - i)$, by definition $\pi = \pi_{1} \cup \pi_{2} \cup \cdots \cup \pi_{2n - i}$
and $$z^{-|\pi_{1}|} z^{-|\pi_{2}|} \cdots z^{-|\pi_{2n-i }|} = z^{-( |\pi_{1}| + |\pi_{2}| + \cdots + |\pi_{2n-i }|) } = z^{-|\pi|} = z^{-i}$$
where the last equality also follows from the definition of $ \mathcal{P_{O}}(i, 2n - i)$.
Assembling the pieces, we have that the constant term of $S(z)$ is equal to $S_{n,k}$ , proving our first claim.

Our second claim is that the constant term for $T(z)$ is exactly $T_{n,k}$.
Observe that
$$T(z) =  \sum_{j=1}^{n-1}  M^{(k)}_{2j}(z - G_{\nu_{k-1}}(z))(z + G_{\nu_{k-1}}(z))^{2(n-j) - 1}.$$
We can immediately discard the $j= n-1$ term since this is equal to
$$ M^{(k)}_{2(n-1)}(z - G_{\nu_{k-1}}(z))(z + G_{\nu_{k-1}}(z)) = M^{(k)}_{2(n-1)}(z^{2} - G_{\nu_{k-1}}(z)^{2}) $$
and this has no constant term.

Isolating a single term for fixed $j$, we have that
\begin{align}
M^{(k)}_{2j}&(z - G_{\nu_{k-1}}(z))(z + G_{\nu_{k-1}}(z))^{2(n-j) - 1} \nonumber \\ & = M^{(k)}_{2j}(z - G_{\nu_{k-1}}(z)) \sum_{p=0}^{2(n-j) - 1} \binom{2(n-j) - 1}{p}  z^{p}G_{\nu_{k-1}}(z)^{2(n-j) - 1 - p } \nonumber \\ \label{Targ1}
&= M^{(k)}_{2j} \sum_{p=0}^{2(n-j) - 1} \binom{2(n-j) - 1}{p} \left[ z^{p+1}G_{\nu_{k-1}}(z)^{2(n-j) - 1 - p } - z^{p}G_{\nu_{k-1}}(z)^{2(n-j)  - p } \right]
\end{align}
Defining $R_{p,j}(z) := z^{p}G_{\nu_{k-1}}(z)^{2(n-j)  - p }$
we isolate a single term for fixed $p,j$
$$  z^{p+1}G_{\nu_{k-1}}(z)^{2(n-j) - 1 - p } - z^{p}G_{\nu_{k-1}}(z)^{2(n-j)  - p } = R_{p+1,j}(z) - R_{p,j}(z).$$
As in the case with $S(z)$, this will have no constant term unless $p+1 \geq 2(n-j)-(p+1)$ so that we may restrict the range in \ref{Targ1} to $p \geq n-j-1$ (for $p = n-j-1$, the term $R_{p+1 , j}(z)$ is generally non-zero whereas the constant term of $R_{p, j}(z)$ is equal to $ 0$  .
Moreover, also arguing as in the case of $S(z)$, we see that the constant term of $R_{p,j}(z)$ is equal to
$$\sum_{\pi \in \mathcal{P_{O}}(p, 2(n-j) - p)} m_{|\pi_{1}| - 1}^{(k-1)} m_{|\pi_{2}| - 1}^{(k-1)}  \cdots
m_{|\pi_{ 2(n-j) - p}| - 1}^{(k-1)} = R_{p,j,n,k} .$$
Putting the pieces together, we have proven our second claim and, therefore, the theorem.
\end{proof}

\section{Concrete Examples}\label{section5}
We now establish additional concrete results based on the Theorems proven in the previous sections.
We begin by calculating the values for $|\mathcal{TCNC}_{2}^{k,k}(n)|$ based on the recursive algorithm in Theorem \ref{momentformula}.

\begin{equation}\label{TCNCarray}
\begin{array}{c|c|c|c|c|c|c}
\ & n=2 & n=4 & n=6 & n=8 & n=10 & n=12 \\
\hline
k=2  & 2 & 6 & 20 & 70 & 252 & 924 \\
\hline
k=3  & 2 & 8 & 38 & 196 & 1062 & 5948  \\
\hline
k=4 & 2 & 8 & 40 & 222 & 1308 & 8014 \\
\hline
k=5  & 2 & 8 & 40 & 224  & 1342 & 8404 \\
\hline
k=6 & 2 & 8 & 40 & 224  & 1344 & 8446 \\
\hline
k > 6 & 2 & 8 & 40 & 224  & 1344 & 8448 \\
\end{array}
\end{equation}

Going through one of the computations that drives Theorem \ref{momentformula}, we consider $|\mathcal{TCNC}_{2}^{2,2}(6)| = M_{3}^{(2)}$.
Utilizing the same reasoning from equalities \eqref{Limit1} through \eqref{Limit3}, we have that this moment is equal to
\[
\begin{split}
\lim_{|z| \uparrow \infty}
& (z - G_{\nu_{1}}(z)) (z + G_{\nu_{1}}(z))^{6} \\
& \cdot \left(G_{\mu_{2,2}}(z + G_{\nu_{1}}(z)) - \frac{1}{z + G_{\nu_{1}}(z)} - \frac{2}{(z + G_{\nu_{1}}(z))^{3}} - \frac{6}{(z + G_{\nu_{1}}(z))^{5}} \right).
\end{split}
\]
Recalling that $G_{\mu_{2,2}}(z + G_{\nu_{1}}(z)) = z - G_{\nu_{1}}(z)$, we distribute these terms,
\begin{align}
M_{3}^{(2)} =& \lim_{|z| \uparrow \infty}   (z + G_{\nu_{1}}(z))^{6} - (z - G_{\nu_{1}}(z)) (z + G_{\nu_{1}}(z))^{5} \nonumber \\ - &  2 (z - G_{\nu_{1}}(z)) (z + G_{\nu_{1}}(z))^{3} - 6(z - G_{\nu_{1}}(z)) (z + G_{\nu_{1}}(z))
\end{align}
We need only isolate the constant terms.  Once again, $(z - G_{\nu_{1}}(z)) (z + G_{\nu_{1}}(z))$ contributes nothing.
Consider
\begin{align}
(z + G_{\nu_{1}}(z))^{6} - & (z - G_{\nu_{1}}(z)) (z + G_{\nu_{1}}(z))^{5} =  2 G_{\nu_{1}}(z) (z + G_{\nu_{1}}(z))^{5} \\
&=  2 G_{\nu_{1}}(z) (z^{5} + 10z^{4}G_{\nu_{1}}(z)) + 10z^{3}G_{\nu_{1}}^{2}(z) + \cdots)
\end{align}
and note that the  $\cdots$ terms make no contribution to the constant as their degree is too low.
The constant term is equal to $2 [ m^{(1)}_{4} + 5(2m_{0}^{(1)}m_{2}^{(1)}) + 10(m_{0}^{(1)})^{3}]$.
By a similar argument, the term
$$ 2 (z - G_{\nu_{2}}(z)) (z + G_{\nu_{2}}(z))^{3}$$ contributes
$-2(2m_{2}^{(1)})$ to the constant. Now, since $\nu_{1} = \delta_{0}$, we have $m_{0}^{(1)} = 1$ and $m_{i}^{(1)} = 0$ for all $i > 0$.
Thus, the only contributing term is $20(m_{0}^{(1)})^{3} = 20$, matching Figure $1$ and our table above.


\begin{Example}\label{simplest}
We isolate a special case of Proposition \ref{Bernoulli} as it is a simple concrete example of a non-commutative convolution that can be computed through the traditional Cauchy transform methodology.

Let $E: M_{2}(\mathbb{C}) \mapsto \mathcal{D}$ denote the non-commutative probability space generated by the conditional expectation of $M_{2}$ onto the diagonal subalgebra.  Let $X = e_{1,2} + e_{2,1}$.  Observe that for $b = \lambda e_{1,1} + \gamma e_{2,2}$ we have that
$$XbX = \alpha(b) = \gamma e_{1,1} + \lambda e_{2,2} \in \mathcal{D}$$
so that the hypotheses of Proposition \ref{Bernoulli} are satisfied.
We let $\mu$ denote the distribution of $X$.

Calculating the various transforms, we have
\begin{equation}
G_{\mu}( b) = \sum_{n=0}^{\infty} \left( \begin{array}{cc}
[\lambda(\gamma \lambda)^{n}]^{-1} & 0 \\
0 & [\gamma( \lambda \gamma)^{n}]^{-1} \\
\end{array} \right) = \left( \begin{array}{cc}
\frac{1}{\lambda - \gamma^{-1}} & 0 \\
0  &  \frac{1}{\gamma - \lambda^{-1}} \\
\end{array} \right)
\end{equation}
\begin{equation}
 F_{\mu}(b) = \left( \begin{array}{cc}
\lambda - \frac{1}{\gamma} & 0 \\
0 & \gamma - \frac{1}{\lambda} \\
\end{array} \right)
\end{equation}
\begin{equation}
F_{\mu}^{\langle -1 \rangle} (b) = \left( \begin{array}{cc}
\frac{1}{2} \left[\lambda  + \sqrt{ \lambda^{2} + 4\frac{\lambda}{\gamma} } \right] & 0 \\
0 & \frac{1}{2} \left[\gamma  + \sqrt{ \gamma^{2} + 4\frac{\gamma}{\lambda} } \right]  \\
\end{array} \right)
\end{equation}
Utilizing the operator-valued version of linearizing property \eqref{Linearization} proved in \cite{Voiculescu1,Popa-Vinnikov-NC-functions}, that is,
$$F_{\mu \boxplus \mu}^{\langle -1 \rangle} (b) =
 2F_{\mu }^{\langle -1 \rangle} (b) - b  $$
we conclude that
\begin{equation}
F_{\mu \boxplus \mu}^{\langle -1 \rangle} (b) = \left( \begin{array}{cc}
\sqrt{ \lambda^{2} + 4\frac{\lambda}{\gamma} }  & 0 \\
0 & \sqrt{ \gamma^{2} + 4\frac{\gamma}{\lambda} }   \\
\end{array} \right).
\end{equation}
Taking the compositional inverse, we have
\begin{equation}
 F_{\mu \boxplus \mu}(b) = \left( \begin{array}{cc}
\sqrt{\lambda^{2} - \frac{4\lambda}{\gamma}}  & 0 \\
0 & \sqrt{\gamma^{2} - \frac{4\gamma}{\lambda}}  \\ \end{array} \right).
\end{equation}
Letting $\lambda = \gamma = z$, the entries are precisely the $ F$-transform of the arcsine distribution.
This, coupled with observation \eqref{Bernoulli2} in Proposition \ref{Bernoulli} allows to reprove the main result in that proposition from more basic principles in this special case.
\end{Example}

\begin{Example}\label{counterexample}
We construct examples of Jacobi-Szeg\H{o} distributions $\mu_{1}$ and $\mu_{2}$ such that $\mu_{1} \boxplus \mu_{2}$ is not a Jacobi-Szeg\H{o} distribution.

Indeed, let $\mu_{1}$ and $\mu_{2}$ be symmetric Bernoulli distributions with respective morphisms $\alpha_{1}$ and $\alpha_{2}$.  That is,
\begin{equation}
\mu_{i} = J
\begin{pmatrix}
0, & 0, & 0, & 0, & \ldots \\
\alpha_{i} , & 0, & 0, & 0, & \ldots
\end{pmatrix}.
\end{equation}

We assume that
\begin{equation}
\mu = J
\begin{pmatrix}
0, & 0, & 0, & 0, & \ldots \\
\beta_1, & \beta_2, & \beta_3, & \beta_4, & \ldots
\end{pmatrix}.
\end{equation}
satisfies $\mu = \mu_{1} \boxplus \mu_{2}$
and show that $\alpha_{1}$ and $\alpha_{2}$ may be chosen so that this precipitates a contradiction.

By definition of the Jacobi parameters,
\begin{align}
\mu(Xb_{0}X) &= \beta_{1}(b_{0}), \label{auto1} \\
\mu(Xb_{1}Xb_{2}Xb_{3}X) &= \beta_{1}(b_{1} \beta_{2}(b_{2}) b_{3}) + \beta_{1}(b_{1})b_{2}\beta_{1}(b_{3}) \label{auto3}.
\end{align}
On the other hand, according to Theorem \ref{maintheorem},
\begin{align}
  \mu_{1} \boxplus \mu_{2} (Xb_{0}X) &= \alpha_{1}(b_{0}) + \alpha_{2}(b_{0}), \label{auto2} \\
 \mu_{1} \boxplus \mu_{2}(Xb_{1}Xb_{2}Xb_{3}X) &= \alpha_{1}( b_{1}\alpha_{2}(b_{2})b_{3} ) + \alpha_{2}( b_{1}\alpha_{1}(b_{2})b_{3} ) \label{auto4}
+\alpha_{1}(b_{1})b_{2}\alpha_{1}(b_{3}) \\ & \nonumber  + \alpha_{1}(b_{1})b_{2}\alpha_{2}(b_{3}) + \alpha_{2}(b_{1})b_{2}\alpha_{1}(b_{3}) +\alpha_{2}(b_{1})b_{2}\alpha_{2}(b_{3}).
\end{align}

Now, \eqref{auto1} and \eqref{auto2} combine to imply that $\beta_{1} = \alpha_{1} + \alpha_{2}$.  At this point, equality of expressions \eqref{auto3} and \eqref{auto4} becomes completely untenable in most non-commutative settings.  For example, letting $\alpha_{1} = \alpha$ from Example \ref{simplest}, $\alpha_{2} = I$ and $ b_{1} = b_{2} = b_{3} = e_{1,1}$ we obtain an easy contradiction.
\end{Example}



\begin{thebibliography}{ABFN13}
\bibitem[AB98]{AccBozGaussianization}
Luigi Accardi and Marek Bo{\.z}ejko, \emph{Interacting {F}ock spaces and
  {G}aussianization of probability measures}, Infin. Dimens. Anal. Quantum
  Probab. Relat. Top. \textbf{1} (1998), no.~4, 663--670. MR1665281
  (2000d:60158)

\bibitem[Ans03]{AnsMeixner}
Michael Anshelevich, \emph{Free martingale polynomials}, J. Funct. Anal.
  \textbf{201} (2003), no.~1, 228--261. MR1986160 (2004f:46079)

\bibitem[Ans07]{AnsFree-Meixner}
\bysame, \emph{Free {M}eixner states}, Comm. Math. Phys. \textbf{276} (2007),
  no.~3, 863--899. MR2350440 (2009b:81106)

\bibitem[Ans10]{Ans-Product}
\bysame, \emph{Product-type non-commutative polynomial states}, Noncommutative
  harmonic analysis with applications to probability {II}, Banach Center Publ.,
  vol.~89, Polish Acad. Sci. Inst. Math., Warsaw, 2010, pp.~45--59.
  MR2730861

\bibitem[ABFN13]{Ans-Bel-Fev-Nica}
Michael Anshelevich, Serban~T. Belinschi, Maxime Fevrier, and Alexandru Nica,
  \emph{Convolution powers in the operator-valued framework}, Trans. Amer.
  Math. Soc. \textbf{365} (2013), no.~4, 2063--2097. MR3009653

\bibitem[BPV13]{Belinschi-Popa-Vinnikov-Semicircle}
S.~T. Belinschi, M.~Popa, and V.~Vinnikov, \emph{On the operator-valued
  analogues of the semicircle, arcsine and {B}ernoulli laws}, J. Operator
  Theory \textbf{70} (2013), no.~1, 239--258. MR3085826

\bibitem[BW01]{Boz-Wys}
Marek Bo{\.z}ejko and Janusz Wysocza{\'n}ski, \emph{Remarks on
  {$t$}-transformations of measures and convolutions}, Ann. Inst. H. Poincar\'e
  Probab. Statist. \textbf{37} (2001), no.~6, 737--761. MR1863276
  (2002i:60005)

\bibitem[Chi78]{Chihara-book}
T.~S. Chihara, \emph{An introduction to orthogonal polynomials}, Gordon and
  Breach Science Publishers, New York, 1978, Mathematics and its Applications,
  Vol. 13. MR0481884 (58 \#1979)

\bibitem[Fla80]{Flajolet}
P. Flajolet, \emph{Combinatorial aspects of continued fractions},
  Discrete Math. \textbf{32} (1980), no.~2, 125--161. MR592851 (82f:05002a)

\bibitem[HRFS07]{Helton-Far-Speicher}
J.~William Helton, Reza Rashidi~Far, and Roland Speicher, \emph{Operator-valued
  semicircular elements: solving a quadratic matrix equation with positivity
  constraints}, Int. Math. Res. Not. IMRN (2007), no.~22, Art. ID rnm086, 15.
  MR2376207 (2008k:15017)

\bibitem[M{\l}o09a]{Mlotkowski-Cumulants-Jacobi}
Wojciech M{\l}otkowski, \emph{Combinatorial relation between free cumulants and
  {J}acobi parameters}, Infin. Dimens. Anal. Quantum Probab. Relat. Top.
  \textbf{12} (2009), no.~2, 291--306. MR2541398

\bibitem[PV13]{Popa-Vinnikov-NC-functions}
Mihai Popa and Victor Vinnikov, \emph{Non-commutative functions and the
  non-commutative free {L}\'evy-{H}in\v cin formula}, Adv. Math. \textbf{236}
  (2013), 131--157. MR3019719

\bibitem[Spe98]{SpeHab}
Roland Speicher, \emph{Combinatorial theory of the free product with
  amalgamation and operator-valued free probability theory}, Mem. Amer. Math.
  Soc. \textbf{132} (1998), no.~627, x+88. MR1407898 (98i:46071)

\bibitem[Vie84]{Viennot-Notes}
G{\'e}rard Viennot, \emph{Une th{\'e}orie combinatoire des polyn{\^o}mes
  orthogonaux g{\'e}n{\'e}raux}, Univ. Quebec, Montreal, Que. (unpublished
  notes), 1984.

\bibitem[VDN92]{VDN}
D.~V. Voiculescu, K.~J. Dykema, and A.~Nica, \emph{Free random variables}, CRM
  Monograph Series, vol.~1, American Mathematical Society, Providence, RI,
  1992, A noncommutative probability approach to free products with
  applications to random matrices, operator algebras and harmonic analysis on
  free groups. MR1217253 (94c:46133)

\bibitem[Voi95]{Voiculescu1}
Dan-Virgil Voiculescu, \emph{Operations on certain non-commutative operator-valued random
              variables}, Recent advances in operator algebras (Orl{\'e}ans, 1992), Ast\'erisque \textbf{232}
  (1995), 243--275. MR1372537 (97b:46081)


\end{thebibliography}

\def\cprime{$'$} \def\cprime{$'$}
\providecommand{\bysame}{\leavevmode\hbox to3em{\hrulefill}\thinspace}
\providecommand{\MRhref}[2]{%
  \href{http://www.ams.org/mathscinet-getitem?mr=#1}{#2}
}
\providecommand{\href}[2]{#2}

\end{document}